\documentclass[11pt,a4paper]{article}
\usepackage{graphicx}
\usepackage{color}
\usepackage{amsmath,amsthm,amsfonts}

\usepackage[utf8]{inputenc}

\newcommand\conv{\mathop{\mathrm{conv}}}
\newcommand\Borel{E}

\newcommand\R{\mathbb{R}}
\newcommand\C{\mathbb{C}}
\newcommand\Z{\mathbb{Z}}
\newcommand\N{\mathbb{N}}

\newcommand\calH{\mathcal{H}}

\newcommand\calL{\mathcal{L}}

\newcommand\weakto{\mathop{\rightharpoonup}}
\newcommand\e{\varepsilon}

\newcommand\self{\mathrm{\mathrm{self}}}
\newcommand\supp{\mathop{\mathrm{supp}}}

\newcommand\dist{{\mathrm{dist}}}
\newcommand\rel{{\mathrm{rel}}}

\newcommand\Id{{\mathrm{Id}}}

\newcommand\sym{{\mathrm{sym}}}

\newcommand\loc{{\mathrm{loc}}}
\newcommand\eps{{\varepsilon}}

\newcommand{\LM}[1]{\hbox{\vrule width.2pt \vbox to#1pt{\vfill \hrule width#1pt height.2pt}}}
\newcommand{\LL}{{\mathchoice
{\,\LM7\,}{\,\LM7\,}{\,\LM5\,}{\,\LM{3.35}\,}}}
\newtheorem{theorem}{Theorem}[section]
\newtheorem{lemma}[theorem]{Lemma}

\newtheorem{remark}[theorem]{Remark}

\newtheorem{proposition}[theorem]{Proposition}
\numberwithin{equation}{section}
\usepackage{fancyhdr}
\begin{document}
\begin{center}
{ \LARGE
Derivation of strain-gradient plasticity 
from a generalized Peierls-Nabarro model\\[5mm]}
{\today}\\[5mm]
Sergio Conti,$^{1}$ Adriana Garroni,$^{2}$ and Stefan M\"uller$^{1,3}$
\\[2mm]
{\em $^1$ Institut f\"ur Angewandte Mathematik,
Universit\"at Bonn\\ 53115 Bonn, Germany }\\
{\em $^{2}$ Dipartimento di Matematica, Sapienza, Universit\`a di Roma\\
00185 Roma, Italy}\\
{\em   $^3$ Hausdorff Center for Mathematics, Universit\"at Bonn\\
    53115 Bonn, Germany}\\[4mm]

\begin{minipage}[c]{0.8\textwidth}
We derive strain-gradient plasticity from a nonlocal phase-field model of 
dislocations in a plane. Both a continuous energy with linear growth depending 
on a measure which characterizes the macroscopic dislocation density and a 
nonlocal effective energy representing the far-field interaction between 
dislocations arise naturally as scaling limits of the nonlocal elastic 
interaction. Relaxation and formation of microstructures at intermediate scales 
are automatically incorporated in the limiting procedure based on 
$\Gamma$-convergence.
\end{minipage}
\end{center}

\section{Introduction}

Crystal plasticity and dislocations are a fundamental theme in the mechanics of solids. Whereas the calculus of variations has been very helpful in the study of nonlinear elasticity and phase transitions, the study of plasticity and dislocations has proven much more demanding. Dislocations are topological singularities of the strain field, which share many features with other important classes of topological defects, such as Ginzburg-Landau vortices, defects in liquid crystals, harmonic maps, models of superconductivity. Their importance for the understanding of the yield behavior of crystals motivated a large interest, and indeed 
in the last decade tools have been developed to study individual dislocations, both in reduced two-dimensional formulations in which one deals with point singularities \cite{Ponsiglione2007,AlicandroCicalesePonsiglione2011,DelucaGarroniPonsiglione2012}, and, with some geometrical restrictions, in the three dimensional setting in which singularities are lines \cite{ContiGarroniOrtiz2015}.
{In this paper we go beyond the scale of single dislocation lines and derive macroscopic strain-gradient plasticity, starting from 
a Peierls-Nabarro model. The limiting model 
combines effects at different scales, it
includes both long-range interactions of singularities and a short-range strain-gradient term which arises from the self-interaction of dislocations. This result does not require any 
restriction on the admissible configurations of defects. One important tool is a self-similar multiscale decomposition of the interaction kernel.}

In order to relate the presence of dislocations to macroscopic plastic deformations of solids, in particular within the framework of a strain-gradient theory of plasticity, 
one needs to consider the situation in which many dislocations are present and their {density} couples to macroscopic deformation of the material. 
This scaling was first addressed in the context of Ginzburg-Landau vortices by Sandier and Serfaty \cite{SandierSerfaty2003}. For point dislocations in two dimensions
this regime was  first studied in \cite{GarroniLeoniPonsiglione2010} 
in a dilute geometrically linear setting, and then generalized 
to geometrically nonlinear models \cite{MuellerScardiaZeppieri,MuellerScardiaZeppieri2015} and to the non-dilute setting \cite{Ginster1,Ginster2}.
A model based on microrotations was presented and studied in two dimensions in \cite{LauteriLuckhaus2016}, where in particular 
the Read-Shockley formula for the energy of small-angle grain boundaries, based on the appearence of an array of point dislocations, was derived.

The three-dimensional case is substantially more subtle.  Firstly, because the geometry plays an important role in the interaction between line singularities, and secondly, because it is a higher-order tensorial problem, in which the energy only controls some components of the relevant strain field. Moreover, the anisotropy of the energy leads to additional relaxation at intermediate scales \cite{CacaceGarroni2009,ContiGarroniMueller2011,ContiGarroniOrtiz2015}.

A natural geometrical restriction involves assuming that a single slip plane is 
active, so that the singularities are lines in a prescribed plane. The elastic 
problem can then be solved  (implicitly) and  results in a nonlocal interaction 
between the singularities, which scales as the $H^{1/2}$ norm {of the slip field}. This model, which 
is in the spirit of the classical Peierls-Nabarro model, was proposed and 
studied numerically in 
\cite{Ortiz1999,KoslowskiCuitinoOrtiz2002,KoslowskiOrtiz2004}, in a regime in 
which a few individual dislocations are present. 
In the same regime, a scalar simplification of the model was studied analytically in \cite{GarroniMueller2005,GarroniMueller2006}. 
In the limit of small lattice spacing $\eps$, the energy concentrates along the dislocations and the problem reduces to a line-tension model; 
the relevant independent variable is a  measure concentrated on a line.
One important tool was the study of variational convergence for phase transitions with nonlocal interactions, see also \cite{AlbertiBouchitteSeppecher1994,AlbertiBouchitteSeppecher1998,FocardiGarroni2007,SavinValdinoci2012}. The extension to the physically relevant vectorial situation in \cite{CacaceGarroni2009,ContiGarroniMueller2011} lead to the discovery of unexpected microstructures at intermediate scales, which arise due to the interplay of the localized nature of the singularity with the anisotropy of the vectorial elastic energy. A generalization to multiple planes is discussed in \cite{ContiGladbach2015}.

We address here a scaling regime in which the total length of the dislocations diverges, rendering the average line-tension  energy comparable with the macroscopic long-range elastic energy.
This is a natural scaling in the presence of topological singularities, since it balances the long-range interaction with the short-range core energies. This critical scaling regime was first considered in the context of Ginzburg-Landau vortices and models of superconductivity 
\cite{SandierSerfaty2003,SandierSerfaty2007,BaldoJerrardOrlandiSoner2012,BaldoJerrardOrlandiSoner2013}. Whereas some ideas are closely related, the vectorial and anisotropic nature of the dislocation problem renders a direct transfer difficult and requires new techniques, in particular for treating the microstructures at intermediate scales.

We derive here a macroscopic strain-gradient theory, where the macroscopic 
effect of the dislocations is captured by a dislocation density, which is a 
measure in $H^{-1/2}$. The  mechanical implications of this result and the 
connection to strain-gradient theories of plasticity have been discussed in 
\cite{ContiGarroniMueller2016,ArizaContiGarroniOrtiz2018}. These include 
dislocation microstructures at intermediate scales, in the form of dislocation 
networks, {the study of which has} long been an important problem in mechanics \cite{HirthLothe1968,FleckHutchinson1993,
OrtizRepetto1999,Gurtin2005,groma2007dynamics,acharya2010new,groma2010variational,berdichevsky2016energy,monavari2016continuum}.
We present here a joint convergence result, in which the two terms are derived in one step from the Peierls-Nabarro model in the limit of small lattice spacing. The study of a single limiting process is crucial to obtain the limiting behavior of both the local and the nonlocal term; if the various homogenization and relaxation steps are taken separately then the nonlocal (interaction) term  disappears \cite{ContiGarroniMueller2017-BUMI}.

The main difficulty in the proof is to obtain a joint treatment of the many different scales present in the problem. The discrete nature of the dislocations leads to localization and to {slip fields} in $BV(\Omega;\Z^N)$, at the same time the nonlocal part requires {slip fields} in $H^{1/2}{(\Omega;\R^N)}$. These two spaces are, except for constants, disjoint (see Lemma \ref{lemmah12bv} below), hence  both requirements can only be realized approximately. This is performed introducing a number of {well-separated} scales, regularizations and cutoffs, as discussed in the introduction to Section \ref{sectupper} for the upper bound and Section \ref{sectlowerb} for the lower bound. Indeed, both the original functional and the limiting functional are finite on $H^{1/2}{(\Omega;\R^N)}$, whereas the relaxation steps occur at intermediate scales, where the relevant functions belong to $BV(\Omega;\Z^N)$. 
As it is clear from this summary, our argument makes a strong usage of the 
existence of a lifting of the dislocation density to  a function {of bounded variation}. The 
extension to the unconstrained three-dimensional case will probably require
a different functional framework, which {could possibly} be formulated via 
cartesian currents 
\cite{Hochrainer2013,ContiGarroniMassaccesi2015,ScalaVangoethem2016}.

We now formulate the model we study. 
The total energy associated to a phase field $u\in L^2(\Omega;\R^N)$, with $\Omega\subset\R^2$ open and bounded, is
\begin{align}
 E_\eps[u,\Omega]:=&\frac1\eps \int_\Omega W(u(x))dx\nonumber\\
 &+ \int_{\Omega\times\Omega} \Gamma(x-y) (u(x)-u(y))\cdot (u(x)-u(y))dxdy.
 \label{uno}
\end{align}
The nonlinear potential $W:\R^N\to[0,\infty)$ satisfies
\begin{equation}\label{eqWdist}
 \frac 1c\, 
\dist^2(\xi,\Z^N)\le W(\xi)\le c \, \dist^2(\xi,\Z^N)
 \end{equation}
 for some $c>0$.
The elasticity kernel $\Gamma\in L^1_\loc(\R^2;\R^{N\times N}_\sym)$ is
defined by
\begin{equation}
\label{eq:derGammaintro}
\Gamma(z):=\frac{1}{|z|^{3}} \hat\Gamma\left(\frac{z}{|z|}\right)\,,
\end{equation}
where
 $\R^{N\times N}_\sym$
denotes the set of symmetric $N\times N$ matrices and
$\hat\Gamma\in L^\infty(S^1;\R^{N\times N}_\sym)$ obeys, for some $c>0$,
\begin{equation}\label{eqgammaposdeftheointro}
 \hat\Gamma(z)=\hat\Gamma(-z) \text{ and }
\frac{1}{c}|\xi|^2\le  \hat\Gamma(z)\xi\cdot\xi\le c|\xi|^2 \hskip5mm 
\text{ for all } \xi\in \R^N, \, z\in S^1\,.
\end{equation}
The specific form of $\hat \Gamma$ depends on the elastic constants and the Burgers vectors of the 
crystal, for example for an elastically
isotropic cubic crystal one obtains $\Gamma=\Gamma^\mathrm{cubic}$, where
\begin{equation}\label{eqgammacubic}
\Gamma^\mathrm{cubic}(z):=\frac{\mu}{16\pi(1-\overline{\nu})|z|^3}\left(\begin{matrix}
\overline{\nu}+1-3\overline{\nu}\frac{z_2^2}{|z|^2} &   3\overline{\nu}\frac{z_1z_2}{|z|^2}\\  
3\overline{\nu}\frac{z_1z_2}{|z|^2}&\overline{\nu}+1-3\overline{\nu}\frac{z_1^2}{|z|^2}  \end{matrix} \right)\,.
\end{equation}
In (\ref{eqgammacubic}) $\overline{\nu}$ and $\mu$ denote the 
material's Poisson's ratio and shear modulus, respectively (see \cite{CacaceGarroni2009}).
It is easy to see that for $\mu>0$ and $\overline{\nu}\in (-1,1/2)$ the kernel $ \Gamma^\mathrm{cubic}$ fulfills the assumption (\ref{eqgammaposdeftheointro}).

In order to present our main result we first introduce the several effective energy 
densities which are generated by the rescaling procedure. Detailed explanations on the physical 
significance of the different steps are given in Section \ref{sechomo}.
In a first step the nonlocal kernel $\Gamma$ generates an unrelaxed 
line-tension energy $\psi:\Z^N\times S^1\to {[0,\infty)}$ by
\begin{equation}\label{eqpsifromgamma}
 \psi(b,n):=2\int_{\{x\cdot n=1\}} \Gamma(x)b\cdot b \,d\calH^1(x).
\end{equation}
Relaxation at the line-tension scale leads to the 
$BV$-elliptic envelope $\psi_\rel:\Z^N\times S^1\to[0,\infty)$ of
$\psi:\Z^N\times S^1\to[0,\infty)$, defined by 
\begin{equation}\label{eqdefpsirel}
\begin{split}
 \psi_\rel(b,n):=&\inf \Bigl\{ \frac12\int_{J_u\cap B_1} \psi([u],\nu) 
d\calH^1: 
u\in BV_\loc(\R^2;\Z^N), \\
 &\hskip1.5cm\supp(u-u^0_{b,n})\subset\subset B_1\Bigr\} ,
\end{split}
\end{equation}
where $u^0_{b,n}(x):=b\chi_{\{x\cdot n>0\}}$ and $\nu$ is the normal to the 
jump set $J_u$ of $u$. We recall that the concept of $BV$-elliptic envelope was 
introduced and studied in \cite{AmbrosioBraides1990a,AmbrosioBraides1990b}, see also \cite[Sect.~5.3]{AmbrosioFP}.

Finally, in the second relaxation step one obtains a continuous energy density
$g:\R^{N\times 2}\to {[0,\infty)}$ defined as  
 the convex envelope of 
\begin{equation}\label{eqg0}
 g_0(A):= \begin{cases}
    0, & \text{ if } A=0,\\
    \psi_\rel(b,n), & \text{ if } A=b\otimes n \text{ for } b\in \Z^N, 
n\in S^1,\\
                \infty , & \text{ otherwise.}
               \end{cases}
\end{equation}
In particular, the function $g$ turns out to be  positively 1-homogeneous, see \cite{ContiGarroniMueller2017-BUMI}.

Our main result is the following.

\begin{theorem}\label{theorem1}
 Let $\Omega\subset\R^2$ be a bounded connected Lipschitz domain,
and let $E_\eps[\cdot, \Omega]$ be defined as in (\ref{uno}), with 
$W$ and $\Gamma$ which satisfy 
(\ref{eqWdist})--(\ref{eqgammaposdeftheointro}).

We say that a family of functions $u_\eps\in L^2(\Omega;\R^N)$, $\eps>0$, converges to $u$ if
\begin{equation}\label{eqtheorem1conv}
 \frac{u_\eps}{\ln (1/\eps)} \to u \hskip1cm \text{ in } L^2(\Omega;\R^N) \hskip1cm\text{ as }\eps\to0.
\end{equation}
With respect to this convergence we have
\begin{equation*}
\Gamma\hbox{-}\lim_{\e\to 0} 
\frac{1}{(\ln (1/\e))^2} 
E_{\e}[\cdot,\Omega] = F_0[\cdot,\Omega] \,,
\end{equation*}
where $F_0$ is defined by
\begin{equation}\label{eqdefF}
F_0[u,\Omega]:=F_\self[u,\Omega]+\int_{\Omega\times\Omega} \Gamma(x-y) (u(x)-u(y))\cdot(u(x)-u(y))dxdy 
\end{equation}
and
\begin{equation}\label{eqdefFself}
 F_\self[u,\Omega]:=\int_\Omega g(\nabla u) dx + \int_\Omega g\left(\frac{dD^su}{d|D^su|}\right)d|D^su|
\end{equation}
if $u\in BV(\Omega;\R^N)\cap H^{1/2}(\Omega;\R^N)$, and $F_0[u,\Omega]=\infty$ otherwise. 
Here $g:\R^{N\times 2}\to[0,\infty)$ is the convex envelope of the function $g_0$
defined from the kernel $\Gamma$ in (\ref{eqpsifromgamma})--(\ref{eqg0}).

Further, the functionals $(\ln (1/\eps))^{-2}E_\eps[\cdot,\Omega]$ are, with respect to the stated convergence, equicoercive, in the sense that if $u_\eps\in L^2(\Omega;\R^N)$ are such that
$E_\eps[u_\eps,\Omega]\le C (\ln (1/\eps))^2$ for all $\eps$ then there is 
a subsequence $\eps_k\to0$ such that, for some $d_k\in\Z^N$ and some
 $u\in L^2(\Omega;\R^N)$, one has
 \begin{equation}
  \lim_{k\to\infty} \frac{u_{\eps_k}-d_k}{\ln (1/\eps_k)} = u \hskip1cm \text{ in } L^2(\Omega;\R^N).
 \end{equation}
\end{theorem}
We remark that $\psi_\rel$ coincides with the line-tension energy density 
obtained in the subcritical regime \cite{ContiGarroniMueller2011}, see Theorem 
\ref{theo:linetension} below. 

\medskip

The proof of above theorem is a combination of various results proved in the rest of the paper. The compactness assertion follows from Proposition \ref{propcompactness} in Section \ref{sectcpt}, 
the upper bound from 
Proposition \ref{propgammalimsup} in Section \ref{sectupper}, and the 
lower bound from Proposition \ref{proplowerbound} in Section \ref{sectlowerb}.

\medskip

In closing this Introduction we briefly recall the connection between $E_\eps$ and the classical Peierls-Nabarro model which contains the elastic energy over a three-dimensional domain.
The 
Peierls-Nabarro model, as generalized in \cite{Ortiz1999,KoslowskiCuitinoOrtiz2002,KoslowskiOrtiz2004} to three dimensions, 
expresses the free energy in terms of the slip $v:\Omega\to\R^3$ as
$$ 
E_{\rm free}[v]:= E_{\rm elastic}[v] + E_{\rm interfacial}[v]\,.
$$
Here the first term represents the long-range elastic distortion due to the
slip and the second term penalizes slips that are not integer multiples of the
Burgers vectors of the crystal lattice. One denotes by
${b}_1$, ..., ${b}_N$ the relevant Burgers vectors and considers slips of the form
$$
v=u_1{b}_1+...+u_N{b}_N
$$
where $u:\Omega\subset\R^2\to \R^N$. 
The term $E_{\rm interfacial}$  penalizes values of $u$ far from $\Z^N$, so that $v$ is close to the lattice generated by $\{b_1, \dots, b_N\}$. A simple model is
\begin{equation*}
 E_{\rm interfacial}[v]=\frac1\eps\int_\Omega \dist^2(u,\Z^N) dx,
\end{equation*}
where $u$ is related to $v$ as stated above.
We observe that the specific functional form does not contribute to the limit. 
At variance with many classical results in $\Gamma$-convergence for phase-field models of phase transitions there is no 
equipartition of energy, and the only role of the interfacial energy is to force 
$u$ to jump on a scale $\eps$. The limiting energy arises then completely from 
the elastic term, as it is apparent from the characterization of $g$ and $\psi$ in 
terms of the kernel $\Gamma$ in (\ref{eqpsifromgamma})--(\ref{eqg0}).

The elastic interaction is given by 
\begin{equation*}
  E_{\rm elastic}[v]=\inf \left\{\int_{\Omega\times \R} \frac12\C\nabla U\cdot \nabla U\, dx\right\},
\end{equation*}
where the displacement $U:\Omega\times\R\to\R^3$ is required to
have a discontinuity of  $v=\sum u_ib_i$ across
$\Omega\times\{x_3=0\}$.
Minimizing out $U$ leads to a nonlocal functional of $u$ of the kind of 
(\ref{uno}) up to boundary effects which do
not influence the leading-order behavior, see 
\cite{GarroniMueller2005,GarroniMueller2006} for a discussion.
The factor $\eps$ in $E_{\rm interfacial}$ is proportional to the lattice spacing and arises from the different scaling of the bulk and the interfacial term.
We refer to \cite{GarroniMueller2005,ContiGarroniMueller2016} for a more detailed discussion of this relation.

\begin{remark}
We discuss in this paper the case that $\Omega$ is a bounded Lipschitz set. 
Similar results can be obtained, with the same proofs, for the case that 
$\Omega$ is a torus. In the latter case it is easy to see that the elastic 
energy $E_{\rm elastic}$ coincides with the nonlocal energy in 
$E_\eps$, up to lower-order terms which are continuous in the topology 
considered here. This leads to the model described in 
\cite{ContiGarroniMueller2016}.
\end{remark}

\def\ELT{E^\mathrm{LT}}
\def\ELTrel{E^\mathrm{LT,rel}}
\section{Limits at separated scales}
\label{sechomo}

{In this section we briefly review two previous results on different scalings which have been mentioned in the introduction and that will be used in the proofs.}

If a sequence $u_\eps$ has energy proportional to $\ln (1/\eps)$, in the sense that $E_\eps[u_\eps,\Omega]\le C \ln (1/\eps)$, then asymptotically $u_\eps$ describes 
dislocations with finite total length, and the limiting energy is given by an integral over the line. This is called the line-tension approximation 
and was studied in \cite{GarroniMueller2005,ContiGarroniMueller2011}.
\begin{theorem}[\protect{\cite[Th.~1.1]{ContiGarroniMueller2011}}]\label{theo:linetension}
 Let $\Omega\subset\R^2$ be a bounded connected Lipschitz domain,
and let $E_\eps[\cdot, \Omega]$ be defined as in (\ref{uno}), with 
$W$ and $\Gamma$ which satisfy 
(\ref{eqWdist})--(\ref{eqgammaposdeftheointro}).

The functionals $1/(\ln (1/\e)) 
E_{\e}[\cdot,\Omega]$ are $L^1$-equicoercive, in the sense that if $
E_\e[u_\e,\Omega]\le C \ln (1/\e)$ then there are $d_\e\in \Z^N$ and $u\in 
BV(\Omega;\Z^N)$ such that $u_\e-d_\e$ has a subsequence that converges to $u$ in $L^1(\Omega;\R^N)$.

Further, we have
\begin{equation*}
\Gamma\hbox{-}\lim_{\e\to 0} 
\frac{1}{\ln (1/\e)} 
E_{\e}[\cdot,\Omega] = \ELTrel[\cdot,\Omega] \,,
\end{equation*}
with respect to strong $L^1$ convergence,
where the relaxed line-tension fuctional $\ELTrel$ is defined by
\begin{equation}\label{eqdefeltrel1}
 \ELTrel[u,\Omega]:=
 \begin{cases}
 \displaystyle\int_{J_u\cap\Omega} \psi_\rel\left([u], \nu\right) d\calH^1,
  & \text{ if } u\in SBV(\Omega;  \Z^N),\\
  \infty, & \text{ otherwise,}
 \end{cases}
\end{equation}
with $\psi_\rel$ obtained from $\Gamma$ as in 
(\ref{eqpsifromgamma}) and (\ref{eqdefpsirel}).
\end{theorem}

{In (\ref{eqdefeltrel1}) and in the rest of this paper we use standard notation on functions of bounded variations, which we now briefly recall.
The elements of $BV(\Omega;\R^N)$, $\Omega\subset\R^2$, are the functions in $L^1(\Omega;\R^N)$ whose distributional derivative $Du$ is a bounded measure on $\Omega$  ($D^su$ denotes the part of this measure which is singular with respect to the Lebesgue measure).
In turn, $SBV(\Omega;\R^N)$ denotes the space of special functions of bounded variation, which are the functions in  $BV(\Omega;\R^N)$ whose distributional gradient can be characterized as $Du=\nabla u \calL^2 + [u]\otimes \nu\calH^1\LL J_u$. Here $J_u$ is a 1-rectifiable set called the jump set of $u$, and it is {defined as} the set of points for which $u$ does not have an approximate limit. The normal to this set is denoted by $\nu$,
and $[u]=u^+-u^-$ denotes the jump of the function $u$ across the set $J_u$. For any $\calH^1$-a.e. $x\in J_u$ one has
\begin{equation*}
 \lim_{r\to0} \frac{\calH^1(J_u\cap B_r(x))}{2r}=1
 \end{equation*}
 and
 \begin{equation*}
 \lim_{r\to0} \frac{1}{r^2}\int_{B_r(x)\cap \{\pm(y-x)\cdot\nu>0\}}|u-u^\pm(x)|dy=0.
 \end{equation*}
We refer to \cite{AmbrosioFP} for details.
 }

The main difficulty in the proof of Theorem \ref{theo:linetension} is the fact 
that this problem has no natural rescaling, since infinitely many scales 
asymptotically contribute to the energy. The proof is based on a dyadic 
decomposition of the interaction kernel, which is also used in Section 
\ref{sectlowerb} below, and on an iterative mollification technique which 
permits to show that microstructure can only appear at few scales, and 
therefore that on the average scale there is no microstructure. 
This permits to pass from the nonlocal functional $E_\eps$ to a line-tension functional with the  unrelaxed energy $\psi$; the relaxation from $\psi$ to $\psi_\rel$ takes then place at the line-tension level and does not couple to the nonlocality of $E_\eps$.
We refer to \cite{ContiGarroniMueller2011} for details and we remark that a similar 
formula for the subcritical regime also holds  without the 
geometric restriction to a single plane, if the dislocations are 
dilute \cite{ContiGarroniOrtiz2015}.

For later reference we observe that 
(\ref{eqgammaposdeftheointro}) and the definition in
(\ref{eqpsifromgamma}) 
imply
that $c|b|\le \psi(b,n)$ for all $b\in\Z^N$, $n\in S^1$, and by
(\ref{eqdefpsirel}) we obtain with Jensen's inequality
\begin{equation}\label{eqpsipsirelcoerc}
 c|b|\le \psi_\rel(b,n)\le \psi(b,n) \hskip5mm \text{ for all $b\in\Z^N$, $n\in S^1$}.
\end{equation}

The transition from scaled  line-tension functionals to a functional with a continuous distribution of dislocations was studied in 
 \cite{ContiGarroniMueller2017-BUMI}. Here two effects are present. Firstly, by 
the rescaling the discrete nature of the dislocations is lost, and 
macroscopically one only sees the effective dislocation density, passing from 
$\psi_\rel$ to $g_0$. This corresponds to recovering continuous slips from 
superposition of many atomic-scale plastic slips, and naturally relates to 
strain-gradient plasticity models. Secondly, and already at the macroscopic 
scale, one relaxes $g_0$ (which is finite only on {certain} rank-one matrices) to 
the macrosopic energy $g$. As usual in problems with  linear growth, the gradient constraint 
does not affect the effective energy density which turns out to be convex, 
 see \cite{KirchheimKristensen2016} for a general statement.
 
 \begin{theorem}\label{theo:homog}
  Let $\Omega\subset\R^2$ be a bounded connected Lipschitz domain,
and let $\psi:\Z^N\times S^1\to[0,\infty)$  obey $\frac1c |b|\le \psi(b,n)$ 
for all $b\in\Z^N$ and $n\in S^1$. The functionals
\begin{equation}\label{eqdefeltno1}
 \ELT_\sigma[u,\Omega]:=
 \begin{cases}
 \displaystyle\int_{J_u\cap\Omega} \sigma \psi\left(\frac{[u]}{\sigma}, \nu\right) d\calH^1,
  & \text{ if } u\in SBV(\Omega; \sigma \Z^N),\\
  \infty, & \text{ otherwise,}
 \end{cases}
\end{equation}
are equicoercive with respect to the strong $L^1$ topology, in the sense that 
if $\sigma_k>0$, $\sigma_k\to0$ and $\ELT_{\sigma_k}[u_k,\Omega]\le C$ then there are $d_k\in 
\R^N$ and $u\in BV(\Omega;\R^N)$ such that {$u_k-d_k$ has a subsequence that converges to} $u$ in $L^1(\Omega;\R^N)$.

Further,  with respect to this topology, 
\begin{equation*}
\Gamma\hbox{-}\lim_{\sigma\to 0}   \ELT_\sigma[\cdot,\Omega]= \ELT_0[\cdot,\Omega],
\end{equation*}
where
\begin{equation}\label{eqdefelt0}
 \ELT_0[u,\Omega]:=
 \begin{cases}
 \displaystyle\int_{\Omega}  g\left( \frac{dDu}{d|Du|}\right) d|Du|,
  & \text{ if } u\in BV(\Omega;\R^N),\\
  \infty ,& \text{ otherwise.}
 \end{cases}
\end{equation}
 The function $g$ is positively 1-homogeneous,  obeys $\frac1c |A|\le g(A)\le 
c|A|$, and coincides with the convex envelope of $g_0$, which is defined from 
$\psi$ via $\psi_\rel$ as in 
(\ref{eqdefpsirel}) and  (\ref{eqg0}).
\def\no{
 \begin{equation}\label{eqdefg0theo}
 g_0(A):= \begin{cases}
 0 & \text{ if } A=0,\\
                \psi_\rel(b,n) & \text{ if } A=b\otimes n \text{ for } b\in \R^N, n\in S^1\\
                \infty & \text{ otherwise,}
               \end{cases}
\end{equation}
where $\psi_\rel$ is the $BV$-elliptic envelope of $\psi$, as defined in 
(\ref{eqdefpsirel}).
}
\end{theorem}
We observe that in \cite[Eq.~(1.4)]{ContiGarroniMueller2017-BUMI} there is a typo, the integral should be (as in (\ref{eqdefelt0}) above) over $\Omega$, not $J_u$.
\begin{proof}
 This statement reduces to \cite[Theorem 1.1]{ContiGarroniMueller2017-BUMI} in the case that $\psi$ is $BV$-elliptic (i.e., if $\psi=\psi_\rel$), after a change in notation.
 
 In the general case   we observe that for any (fixed) $\sigma>0$ by \cite{ContiGarroniMassaccesi2015}
 the functional 
\begin{equation}\label{eqdefeltrel}
 \ELTrel_\sigma[u,\Omega]:=
 \begin{cases}
 \displaystyle\int_{J_u\cap\Omega} \sigma \psi_\rel\left(\frac{[u]}{\sigma}, \nu\right) d\calH^1,
  & \text{ if } u\in SBV(\Omega; \sigma \Z^N),\\
  \infty, & \text{ otherwise,}
 \end{cases}
\end{equation}
 is the relaxation of $\ELT_\sigma$, and that 
$\psi_\rel$ is $BV$-elliptic and has linear growth, in the sense that
\begin{equation*}
 \frac1c |b| \le \psi_\rel(b,n) \le c|b| \hskip5mm \text{ for all } b\in\Z^N, n\in S^1.
\end{equation*} 
Therefore the $\Gamma$-limit of  the sequence 
$\ELT_\sigma$ is the same as the $\Gamma$-limit of the sequence 
$\ELTrel_\sigma$, we refer to \cite[Prop.~6.11]{Dalmaso1993} for details. 

Finally, as mentioned above \cite[Theorem 1.1]{ContiGarroniMueller2017-BUMI}  implies that 
the sequence $\ELTrel_\sigma$
$\Gamma$-converges to $\ELT_0$. Coercivity is also inherited, since $\psi_\rel\le \psi$. This  concludes the proof.
\end{proof}

 In proving our main result we shall have to take into account both these 
results, but also include the effects of the long-range elastic energy, which 
scales as the squared $H^{1/2}$ norm of $u$. We remark that $H^{1/2}$ is 
singular with respect to the natural spaces of piecewise constant functions 
entering the above results, hence one cannot recover Theorem 
\ref{theorem1} from a direct combination of Theorem \ref{theo:linetension} and 
Theorem \ref{theo:homog}.

\section{Compactness}
\label{sectcpt}
The functions $u_\e/\ln(1/\e)$ belong to the space $H^{1/2}(\Omega;\R^N)$, the limit however will belong also to 
$BV(\Omega;\R^N)$. The key step in the proof of the compactness result is to produce a new sequence of functions,
called $v_\e$ in the proof below, which belong to $BV(\Omega;\Z^N)$ and are close to $u_\e$.

\begin{proposition}[Compactness]\label{propcompactness}
 Let $\Omega\subset\R^2$ be a bounded connected Lipschitz domain,
and let $E_\eps[\cdot, \Omega]$ be defined as in (\ref{uno}), with 
$W$ and $\Gamma$ which satisfy 
(\ref{eqWdist})--(\ref{eqgammaposdeftheointro}).
 
Let $u_\eps$ be a family with $E_\eps[u_\eps,\Omega]\le M (\ln(1/\eps))^2$ for some $M>0$.
Then there are a function $u\in BV(\Omega;\R^N)\cap H^{1/2}(\Omega;\R^N)$,  vectors $d_\eps\in\Z^N$ and a subsequence $\eps_k\to0$ such that
\begin{equation}
 \frac{1}{\ln(1/{\eps_k})} (u_{\eps_k}-d_{\eps_k}) \to u \hskip5mm\text{  in $L^2(\Omega;\R^N)$.}
\end{equation}
\end{proposition}
Before starting the proof we recall the basic definition and some properties of the space $H^{1/2}$. 
We define the homogeneous $H^{1/2}$ seminorm of a measurable function $f:\Omega\to\R$, for  an open set $\Omega\subset\R^2$, by
\begin{equation}\label{eqdefh12}
 [f]_{H^{1/2}(\Omega)}^2 := \int_{\Omega\times\Omega} \frac{ |f(x)-f(y)|^2}{|x-y|^3} dx dy.
\end{equation}
We observe that if $\Omega$ is bounded then for any $f$ there is $a_f\in\R$ such that
\begin{equation*}
 \| f- a_f\|_{L^2(\Omega)} \le c_\Omega [f]_{H^{1/2}(\Omega)}.
\end{equation*}
This can be proven directly from the definition
of $[f]_{H^{1/2}(\Omega)}$, letting $a_f$ be the average of $f$ over $\Omega$.
If $\Omega$ is bounded and Lipschitz, then any sequence $f_j$ which converges 
weakly in $L^2$ and is bounded in the $H^{1/2}$ seminorm converges strongly in 
$L^2$, see for example 
\cite[Section 7]{ValdinociFractional}. One can see that this norm is equivalent the one obtained by the trace method.

We {next} recall that $BV(\Omega;\Z^N)$ is (up to constants) disjoint from $H^{1/2}(\Omega;\R^N)$, which only contains functions that ``do not jump''. This fact is made precise by the following Lemma.
\begin{lemma}\label{lemmah12bv}
 Let $\Omega\subset\R^2$ be open,  $u\in BV(\Omega)\cap H^{1/2}(\Omega)$. Then $\calH^1(J_u)=0$.
\end{lemma}
{In particular, if $u\in BV(\Omega;\Z^N)$ then $Du=[u]\otimes \nu\calH^1\LL J_u$, hence in this case $\calH^1(J_u)=0$ implies $Du=0$.}
\begin{proof}
We claim that for any $\delta>0$ we have
\begin{equation}\label{eqh12bvjdelta}
 \calH^1(\{x\in J_u: |u^+-u^-|(x)>\delta\})=0.
\end{equation}
Since $\delta$ is arbitrary, this will imply the assertion.

We write $J^{(\delta)}_u:=\{x\in J_u: |u^+-u^-|(x)>\delta\}$.
Fix $\eta>0$, and choose $\rho>0$ such that
\begin{equation}\label{eqh12bveta}
 \int_{\Omega\times \Omega} \frac{|u(x)-u(y)|^2}{|x-y|^3} \chi_{B_\rho(x)}(y) dxdy\le \eta.
\end{equation}
This is possible, since $u\in H^{1/2}(\Omega)$ and by dominated convergence this integral converges to zero as $\rho\to0$.
For $\calH^1$-a. e. $x\in J^{(\delta)}_u$ one has
\begin{equation*}
 \lim_{r\to0} \frac{1}{r^2} \int_{B_r(x)} |u(y)-u_x(y)| dy =0,
\end{equation*}
where 
\begin{equation*}
u_x(y):=\begin{cases}
         u^+(x),& \text{ if } (y-x)\cdot \nu>0\,,\\
         u^-(x),& \text{ if } (y-x)\cdot \nu\le0\,,
        \end{cases}
      \end{equation*}
with $u^\pm(x)$ the traces and $\nu$ the normal to the jump set in $x$. 
Therefore for $\calH^1$-a. e. $x\in J^{(\delta)}_u$
\begin{eqnarray*}
 \lim_{r\to0}  \frac{1}{\pi r^2} \int_{B_r(x)} |u(y)-(u)_r| dy &
 \\ =\lim_{r\to0}  \frac{1}{\pi r^2} \int_{B_r(x)} |u_x(y)-(u_x)_r| dy 
=& \displaystyle\frac12 
|u^+-u^-|(x)>\frac{\delta}{2}
\end{eqnarray*}
where $(f)_r$ denotes the average of $f$ over the ball  $B_r(x)$. Moreover from the 1-rectifiability of $J_u$, and then of $J^{(\delta)}_u$, we have
\begin{equation*}\label{covering2}
\lim_{r\to0} \frac{1}{2r}
\calH^1(J^{(\delta)}_u\cap B_r(x))=1.
\end{equation*}
By Vitali-Besicovich's covering Lemma (see for instance \cite{AmbrosioFP}, Theorem 2.19) we can choose 
countably many disjoint balls $B_j:=B(x_j,r_j)$ such that for all $j$ one has
$r_j\in(0,\rho/2)$, $x_j\in J^{(\delta)}_u$, 
\begin{equation}\label{eql1blowup}
\frac{1}{\pi r_j^2} \int_{B_j} |u(y)-u_j| dy \ge \frac\delta4  
\end{equation}
where $u_j$ is the average of $u$ over $B_j$,  with
$r_j\le \calH^1 {(J^{(\delta)}_u\cap \overline B_j)}\le 3r_j$, and satisfying
$\calH^1{(J^{(\delta)}_u\setminus \cup_j\overline B_j)}=0$.

For each $B_j$ we estimate, as in the proof of 
 Poincar\'e's inequality for $H^{1/2}$, using Jensen's inequality, 
\begin{equation*}
\begin{split}
 \int_{B_j} |u(y)-u_j|^2 dy \le&
 \frac{1}{\pi r_j^2}\int_{B_j\times B_j} |u(y)-u(x)|^2 dxdy \\
 \le&
 \frac{8r_j}{\pi}\int_{B_j\times B_j} \frac{|u(y)-u(x)|^2}{|x-y|^3} dxdy .
\end{split}
\end{equation*}
In particular, recalling (\ref{eql1blowup}),
\begin{equation*}
 \frac{\delta^2 }{16} \pi r_j^2\le  \int_{B_j} |u(y)-u_j|^2 dy \le
 \frac{8r_j}{\pi}\int_{B_j\times B_j} \frac{|u(x)-u(y)|^2}{|x-y|^3} dxdy .
\end{equation*}
We divide by $r_j$, sum over $j$, and obtain
\begin{equation*}
\begin{split}
 \calH^1(J^{(\delta)}_u)&\le \sum_j
\calH^1({J^{(\delta)}_u\cap \overline B_j})\le \sum_j 3r_j \\
&
 \le C_\delta \int_{\cup_j B_j\times B_j}\frac{|u(x)-u(y)|^2}{|x-y|^3} dxdy\le C_\delta\eta,
\end{split}
\end{equation*}
where $C_\delta$ depends only on $\delta$ and in the last step we used $2r_j< \rho$ and (\ref{eqh12bveta}). Since $\eta$ was arbitrary, the proof of (\ref{eqh12bvjdelta}) and therefore of the Lemma is concluded.
\end{proof}

In order to prove the compactness result we recall some notation and a result
from 
\cite{ContiGarroniMueller2011}.
We define the truncated kernels by
\begin{equation*}
 \Gamma_{[0,k]}:=\sum_{i=0}^k \Gamma_i
\end{equation*}
where
\begin{equation*}
 \Gamma_k(x):=
 \begin{cases}
 \hat\Gamma(x/|x|)( 2^{3(k+1)}-2^{3k}), & \text{ if } 0<|x|\le 2^{-k-1},\\
  \hat\Gamma(x/|x|)(|x|^{-3} -2^{3k}), & \text{ if } 2^{-k-1}<|x|\le 2^{-k},\\
  0 , &\text{ if } |x|>2^{-k},
 \end{cases}
\end{equation*}
see Figure \ref{figtrunc} for an illustration, and the corresponding truncated energies by
\begin{equation}\label{eqdefestk}
 E^{*}_k[u,\Omega]:=\int_{\Omega\times\Omega} \Gamma_{k}(x-y) (u(x)-u(y))\cdot (u(x)-u(y)) dxdy.
\end{equation}

\begin{figure}
 \begin{center}
  \includegraphics[width=6cm]{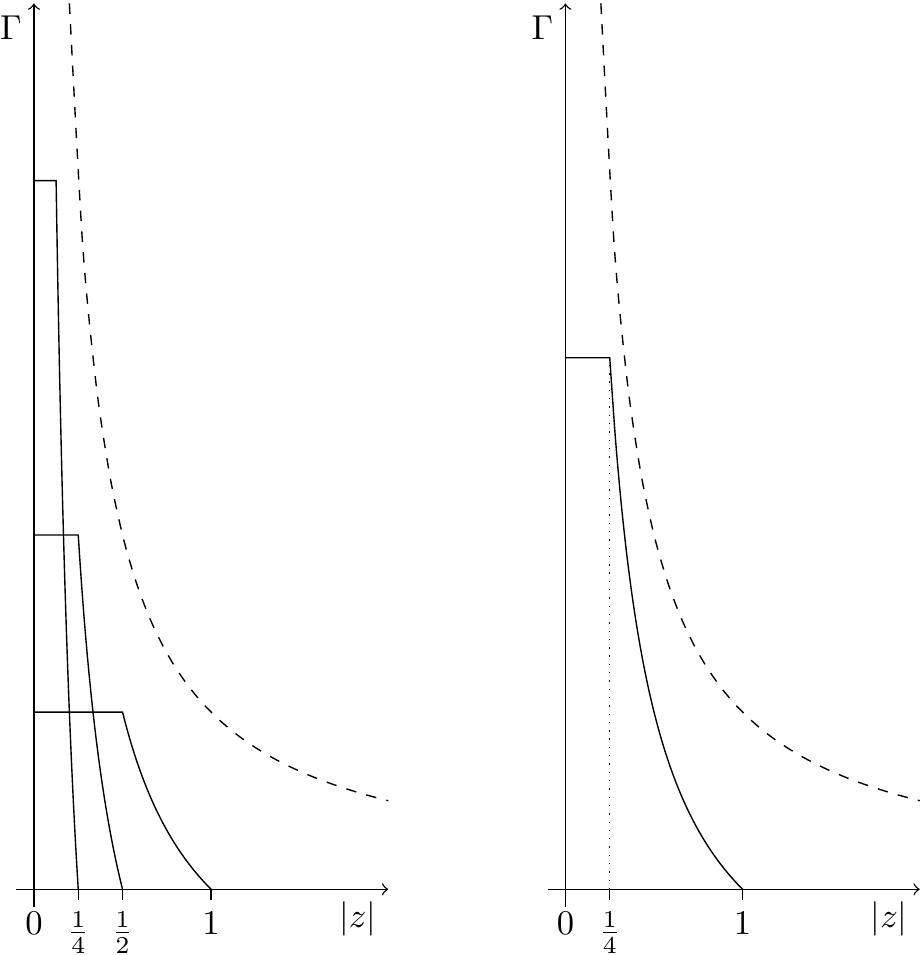}
 \end{center}
\caption{Sketch of $\Gamma$, $\Gamma_0$, $\Gamma_1$ and $\Gamma_2$ (left) and of $\Gamma$, $\Gamma_{[0,2]}$ (right). For clarity we have plotted $1/|z|$ instead of $1/|z|^3$.}
\label{figtrunc} 
\end{figure}

The result from \cite{ContiGarroniMueller2011} that we use concerns the 
approximation of regular fields by $BV$ phase fields.
We observe that the symbol $E_\eps$ is used in 
\cite{ContiGarroniMueller2011} for the energy already divided by 
 $\ln (1/\eps)$, i.e., for the quantity
$E_\eps^{2011}:=E_\eps/(\ln (1/\eps))$.
\begin{proposition}[\protect{\cite[Prop. 4.1]{ContiGarroniMueller2011}}]\label{propmakeubv}
 Let $\Omega\subset\R^2$ be a {bounded Lipschitz domain,}
and let $E_\eps[\cdot, \Omega]$ be defined as in (\ref{uno}), with 
$W$ and $\Gamma$ which satisfy 
(\ref{eqWdist})--(\ref{eqgammaposdeftheointro}).

Assume that 
$\omega\subset\subset\Omega$ and  $\delta\in(0,1/2)$. 
Then there exists a constant $C>0$ such that for every  sufficiently small $\e>0$ (on
a scale set by $\delta$ and $\dist(\omega,\partial\Omega)$) and every
 $u\in L^2(\Omega;\R^N)$ 
there are $k\in\N$ and 
$v\in BV(\omega;\Z^N)$ such that
\begin{equation}\label{eqsumehstar}
\sum_{h=0}^k E_h^*[v,\omega]
\le E_\e[u,\Omega]\left( 1 + \frac{C}{\delta (\ln(1/\e))^{1/2}}\right)\,, 
\end{equation}
\begin{equation}\label{eqpropmakeubvbvnorm}
|Dv|(\omega)\le \frac{C}{\delta} \frac{E_\e[u,\Omega]}{\ln(1/\eps)}\,,
\end{equation}
and 
\begin{equation*}
\e^{1-\delta/2}\le 2^{-k} \le \e^{1-\delta}\,.
\end{equation*}
Furthermore,
\begin{equation*}
\|u-v\|_{L^1(\omega)} \le C 2^{-k/2} \left( \frac{E_\e[u,\Omega]}{\ln(1/\eps)}\right)^{1/2}\,.
\end{equation*}
The constants depend only on $W$ and $\Gamma$.
\end{proposition}

\begin{proof}[Proof of Proposition \ref{propcompactness}]
We start by proving {that the sequence $u_\eps-d_\eps$, for a suitable choice of $d_\eps\in\Z^N$, converges} in $L^2$ to a limit which is contained in $H^{1/2}$. By coercivity of $\Gamma$,
\begin{equation*}
 E_\eps[u_\eps,\Omega] \ge c[u_\eps]_{H^{1/2}(\Omega)}^2.
\end{equation*}
Therefore the sequence $u_\eps/(\ln (1/\eps))$ is bounded in the homogeneous 
$H^{1/2}$ seminorm.  By {the Poincar\'e inequality} we can find vectors $\hat d_\eps\in\R^N$ such that
$(u_\eps-\hat d_\eps)/(\ln (1/\eps))$ is bounded in $H^{1/2}$ and has a subsequence 
which converges 
{weakly in $H^{1/2}$ and}
strongly in $L^2$ to a limit $u$. We choose 
$ d_\eps\in\Z^N$ such that $|d_\eps-\hat d_\eps|\le N^{1/2}$ and observe that
$(d_\eps-\hat d_\eps)/(\ln (1/\eps))\to0$. 

It remains to show that the limit $u$ is in $BV(\Omega;\R^N)$. Let $\omega\subset\subset\Omega$. By Proposition \ref{propmakeubv} with $\delta=1/4$, for sufficiently small $\eps$
there are $k_\eps\in\N$ and $v_\eps\in BV(\omega;\Z^N)$ such that
 $\eps^{7/8}\le 2^{-k_\eps}\le \eps^{3/4}$,
\begin{equation*}
\|v_\eps-u_\eps\|_{L^1(\omega)}\le C M^{1/2} 2^{-k_\eps/2} (\ln (1/\eps))^{1/2} 
\le C M^{1/2}  \eps^{3/8} (\ln (1/\eps))^{1/2} 
\end{equation*}
and
\begin{equation*}
 |Dv_\eps|(\omega)\le c \frac{ E_\eps[u_\eps,\Omega]}{\ln (1/\eps)}\le c M \ln (1/\eps)\,.
\end{equation*}
In particular, after extracting the same subsequence as above, 
$(v_\eps-d_\eps)/\ln (1/\eps)$ converges to $u$ in $L^1(\omega;\R^N)$ and $(v_\eps-d_\eps)/\ln (1/\eps)$ is bounded in $BV(\omega;\R^N)$. 
Therefore, after possibly extracting a further subsequence, we obtain 
\begin{equation*}
 \frac{v_\eps-d_\eps}{\ln (1/\eps)} \weakto u \text{ weakly in } BV(\omega;\R^N)
\end{equation*}
and
\begin{equation*}
 |Du|(\omega)\le c M
\end{equation*}
with $c$ not depending on $\omega$.
Since the bound does not depend on $\omega$ we conclude that $u\in BV(\Omega;\R^N)$.
\end{proof}

\section{Density and approximation}
\label{secedenstiy}
We give here a  refinement of Theorem \ref{theo:homog}, 
that will be needed in the proof of the upper bound. The main difference is that we can approximate with functions which are at the same time 
polyhedral and uniformly bounded in $L^\infty$. This is clearly only possible if the limit is contained in $L^\infty$.
The refined upper bound requires an extra assumption on the energy density which is fulfilled  by the function $\psi$ defined in \eqref{eqpsifromgamma} (see Lemma \ref{lemma64cgo2015newa}).

We recall that $v\in BV_\loc(\Omega;\R^N)$, {$\Omega\subset\R^2$ open}, is polyhedral if $Dv=\sum_{h=0}^H [v_h]\otimes n_h\calH^1\LL  
S_h$, where $H\in\N$,  $S_h=[a_h,b_h]$ is a segment in $\R^2$, and $n_h\in S^1$ is 
normal to $S_h$.

 \begin{proposition}\label{prophomoglinftypolyh}
  Let $\Omega\subset\R^2$ be a bounded Lipschitz domain,
and let $\psi:\Z^N\times S^1\to[0,\infty)$  obey 
$$
\frac1c \, |b|\le \psi(b,n)\leq (1+c|n-n'|) \psi(b,n') \text{ for all } b\in\Z^N, n,n'\in S^1.
$$ 
 For any $u\in L^\infty(\Omega;\R^N)\cap BV(\Omega;\R^N)$, any $\delta>0$, and any sequence 
$\sigma_j\to0$ there is a sequence of polyhedral functions $v_j\in 
SBV(\Omega;\sigma_j\Z^N)$ such that $v_j\to u$ in $L^1$ and
 \begin{equation*}
  \limsup_{j\to\infty} \ELT_{\sigma_j}[v_j,\Omega] \le \ELT_0[u,\Omega] + \delta
 \end{equation*}
with $\sup_j\|v_j\|_{L^\infty(\Omega)}+|Dv_j|(\Omega)<\infty$
{and $\ELT_\sigma$, $\ELT_0$ as in (\ref{eqdefeltno1}) and (\ref{eqdefelt0}).}
\end{proposition}
The proof of Proposition \ref{prophomoglinftypolyh} is based on 
the following density result, that was proven in 
\cite[Lemma.~6.4]{ContiGarroniOrtiz2015}, building on 
 \cite[Corollary 2.2]{ContiGarroniMassaccesi2015}. The key ingredient in this 
construction is the scalar result in \cite[Th. 4.2.20]{Federer}.
 The related 
situation for partition problems was studied in 
 \cite[Theorem 2.1]{BraidesContiGarroni2017}.
\begin{lemma}[\protect{\cite[Lemma.~6.4]{ContiGarroniOrtiz2015}}]
\label{lemma64cgo2015}
 Assume that $\psi:\Z^N\times S^1\to\R$ satisfies
 \begin{equation*}
  \psi(b,n)\le (1+c|n-n'|) \psi(b,n') \text{ for all } b\in\Z^N, n,n'\in S^1.
 \end{equation*}
Assume that $u\in BV(\R^2;\Z^N)$ and let $\Omega\subset\R^2$ be a bounded 
Lipschitz set with $|Du|(\partial\Omega)=0$.

Then for any $\eta\in(0,1)$ there are $r>0$, a polyhedral $v\in BV(\R^2;\Z^N)$ 
and a bijective map $f\in C^1(\R^2;\R^2)$ such that
\begin{equation}\label{eqlemma64cgo2015a}
 |D(u\circ f)-Dv|(\R^2)\le \eta,
\end{equation}
\begin{equation}\label{eqlemma64cgo2015b}
 |Df(x)-\Id|+|f(x)-x|\le\eta \text{ for all }x\in\R^2,
\end{equation}
and
\begin{equation}\label{eqlemma64cgo2015c}
 \int_{J_v\cap\Omega_r} \psi([v],\nu) d\calH^1\le (1+c\eta)  \int_{J_u\cap\Omega} 
\psi([u],\nu) d\calH^1 + c \eta,
\end{equation}
where $\Omega_r:=\{x\in\R^2: \dist(x,\Omega)<r\}$. Further, the restriction of 
$v$ to $\Omega$ is polyhedral.
\end{lemma}
We start by deriving a variant of this Lemma.
\begin{lemma}\label{lemma64cgo2015new}
 Assume that $\psi:\Z^N\times S^1\to\R$ satisfies
 \begin{equation*}
  \psi(b,n)\le (1+c|n-n'|) \psi(b,n') \text{ for all } b\in\Z^N, n,n'\in S^1.
 \end{equation*}
Assume that $u\in BV(\Omega;\sigma \Z^N)$ for some $\sigma>0$ with
$\Omega\subset\R^2$  bounded  and Lipschitz.

Then for any $\eta\in(0,1)$ there is a polyhedral $v\in BV(\Omega;\sigma \Z^N)$ 
such that
\begin{equation}\label{lemma64cgo2015newa}
 |Dv|(\Omega)\le 5 |Du|(\Omega),
\end{equation}
\begin{equation}\label{lemma64cgo2015newb}
 \|u-v\|_{L^1(\Omega)}\le c\sigma + c \eta |Du|(\Omega),
\end{equation}
and
\begin{equation}\label{lemma64cgo2015newc}
 \int_{J_v\cap\Omega} \sigma \psi\left(\frac{[v]}{\sigma},\nu\right) d\calH^1\le (1+c\eta)  
\int_{J_u\cap\Omega} \sigma \psi\left(\frac{[u]}{\sigma},\nu\right) d\calH^1 + c \eta \sigma.
\end{equation}
\end{lemma}
\begin{proof}
 Replacing $u$ by $u/\sigma$ and $v$ by $v/\sigma$ we see that it suffices to 
consider the case $\sigma=1$. We can also assume $|Du|(\Omega)>0$ (otherwise $u$ is constant and $v=u$ will do).
 We extend $u$ to a function $u\in BV(\R^2;\Z^N)$ such that 
$|Du|(\partial\Omega)=0$, for instance, by reflection (see \cite{AmbrosioFP}). 
Possibly reducing $\eta$ we can assume $\eta\le |Du|(\Omega)$ 
 and $|Du|(\Omega_\eta)\le 2|Du|(\Omega)$, where $\Omega_\eta$ is defined as in Lemma \ref{lemma64cgo2015}.
 
 We apply Lemma \ref{lemma64cgo2015} and obtain a polyhedral $v\in BV(\R^2;\Z^N)$ and a diffeomorphism $f$ satisfying 
 (\ref{eqlemma64cgo2015a})--(\ref{eqlemma64cgo2015c}).
  We define
 \begin{equation*}
  d:= \frac{1}{|\Omega|} \int_\Omega (v-u \circ f) \, dx \in\R^N
 \end{equation*}
and choose $\tilde d\in\Z^N$ such that $|d-\tilde d|\le N^{1/2}$. We replace $v$ by 
$v-\tilde d$, so that
 \begin{equation}\label{eqvunum}
  \left|\frac{1}{|\Omega|} \int_\Omega (v-u\circ f) \, dx  \right|\le N^{1/2},
 \end{equation}
while  (\ref{eqlemma64cgo2015a})  and  (\ref{eqlemma64cgo2015c}) are not affected. 
We then estimate using  (\ref{eqlemma64cgo2015a}) and  (\ref{eqlemma64cgo2015b})
 \begin{equation*}
 \begin{split}
  |Dv|(\Omega)\le & \eta + |D(u\circ f)|(\Omega) \\
  \le & \eta + \int_{J_{u\circ f}\cap \Omega} |[u]|\circ f \,d\calH^1\\
  \le & \eta + \int_{J_{u}\cap f(\Omega)} |[u]| {|Df^{-1}\nu^\perp|}\,d\calH^1\\
    \le& \eta + |Du|(\Omega_\eta) \sup\{{|Df^{-1}e|}(x): x\in\R^2, e\in S^1\}\\
    \le & \eta + |Du|(\Omega_\eta) (1+\eta)\le 5 |Du|(\Omega).
 \end{split}
 \end{equation*}
 This proves (\ref{lemma64cgo2015newa}). Since (\ref{lemma64cgo2015newc}) follows immediately from (\ref{eqlemma64cgo2015c}), 
 it remains to prove (\ref{lemma64cgo2015newb}). 
By Poincar\'e, (\ref{eqlemma64cgo2015a}) and (\ref{eqvunum}),
 \begin{equation*}
 \begin{split}
  \|u\circ f-v\|_{L^1(\Omega)} \le N^{1/2}|\Omega| + c|D(u\circ f)-Dv|(\Omega) \le 
N^{1/2} |\Omega|+ c \eta\le c.
 \end{split}
 \end{equation*}
Now if we prove that there is $c_*>0$ such that
  \begin{equation}\label{equucf}
   \|u-u\circ f\|_{L^1(\Omega)} \le 2c_*\eta |Du|(\Omega_\eta),
  \end{equation}
 then with a triangular inequality we obtain $\|v-u\|_{L^1(\Omega)} \le 
c + 2c_*\eta|Du|(\Omega_\eta)$ and conclude the proof.
 
 It remains to prove (\ref{equucf}). We start by proving that for any $z\in C^1(\R^2;\R^N)$ we have
  \begin{equation}\label{eqzzcircfl1}
   \|z-z\circ f\|_{L^1(\Omega)} \le 2c_*\eta |Dz|(\Omega_\eta),
  \end{equation} 
  for some $c_*$ chosen below.
Indeed, 
\begin{equation*}
\begin{split}
 \int_\Omega|z(x)-z(f(x))| dx =&
 \int_\Omega\left|\int_0^1 Dz(x+t(f(x)-x))(f(x)-x) dt\right|dx \\
 \le & \|f(x)-x\|_{L^\infty(\Omega)} \int_0^1\int_\Omega |Dz|(x + t (f(x)-x)) dxdt .\\
\end{split}
\end{equation*}
For any $t\in[0,1]$, we define  $F_t:\R^2\to\R^2$ by $F_t(x):=x+t(f(x)-x)$ and 
estimate by a change of variables and (\ref{eqlemma64cgo2015b})
\begin{equation*}
 \int_\Omega |Dz|(F_t(x)) dx=\int_{F_t(\Omega)} |Dz| {|\det DF_t^{-1}|} dx
 \le \|Dz\|_{L^1(\Omega_\eta)} (1+c_* \eta).
\end{equation*}
Therefore (\ref{eqzzcircfl1}) holds for any $z\in C^1$ and, by density, (\ref{equucf}) is proven.
\end{proof}

In what follows we prove that the unrelaxed line-tension energy density $\psi$ defined in \eqref{eqpsifromgamma} satisfies the assumptions of Lemma  \ref{lemma64cgo2015new}.
\begin{lemma}\label{lemmapsilipschit}
 Let $\Gamma$ obey (\ref{eq:derGammaintro}) and (\ref{eqgammaposdeftheointro}), and let $\psi$ be defined by (\ref{eqpsifromgamma}). Then
 \begin{equation}\label{eqpsibns1}
  \psi(b,n)=\int_{S^1} |y\cdot n| \hat\Gamma(y)b\cdot b\, d\calH^1(y)
 \end{equation}
and
 \begin{equation*}
  \psi(b,n)\le (1+c|n-n'|) \psi(b,n') \text{ for all } b\in\Z^N, n,n'\in S^1.
 \end{equation*}
\end{lemma}
\begin{proof}
To prove the first equality, for a fixed $n\in S^1$, we write both sides in polar coordinates, measuring the angles $\theta$ with respect to the vector $n$.
Precisely,  we let $n^\perp:=(-n_2,n_1)$ and write $y=n\cos\theta+n^\perp\sin\theta$, for $\theta\in[0,2\pi)$. Then
 \begin{align*}
  \int_{S^1} |y\cdot n| \hat\Gamma(y)b\cdot b \,d\calH^1(y)&=
  \int_0^{2\pi} |\cos\theta| \hat\Gamma (n\cos\theta+n^\perp\sin\theta) b \cdot b\, d\theta.
 \end{align*}
At the same time, if $x\cdot n=1$ we write $x=n+t n^\perp$, with $t=\tan\theta$, and using (\ref{eq:derGammaintro}) {and the first condition in (\ref{eqgammaposdeftheointro}),}
\begin{align*}
 2\int_{\{x\cdot n=1\}} \Gamma(x)b\cdot b \,d\calH^1(x)&=
 2\int_{-\infty}^\infty \Gamma(n+n^\perp t ) b\cdot b\,  dt\\
 &=
 2\int_{-\pi/2}^{\pi/2} \Gamma(n+n^\perp\tan\theta ) b\cdot b\, \frac{1}{\cos^2\theta} d\theta\\
 &=
 2\int_{-\pi/2}^{\pi/2} \Gamma\left(\frac{n\cos\theta+ n^\perp\sin\theta}{\cos\theta} \right) b\cdot b\, \frac{1}{\cos^2\theta} d\theta\\
&=
 2\int_{-\pi/2}^{\pi/2} \hat\Gamma(n\cos\theta+ n^\perp\sin\theta ) b\cdot b\, \cos\theta d\theta\\
&=
 \int_{0}^{2\pi} \hat\Gamma(n\cos\theta+ n^\perp\sin\theta ) b\cdot b\, |\cos\theta |d\theta\,.
\end{align*}
This concludes the proof of (\ref{eqpsibns1}).

To prove the estimate we then write, with $|y\cdot n|-|y\cdot n'|\le |y\cdot(n-n')|$ and (\ref{eqgammaposdeftheointro}),
\begin{align*}
 \psi(b,n)-\psi(b,n')&=  \int_{S^1} (|y\cdot n|-|y\cdot n'|) \hat\Gamma(y)b\cdot b\, d\calH^1(y)\\
& \le |n-n'| \int_{S^1}  \hat\Gamma(y)b\cdot b\, d\calH^1(y) \le c |n-n'| |b|^2 
\end{align*}
and, again from  (\ref{eqgammaposdeftheointro}), $|b|^2\le c \psi(b,{n'})$. This concludes the proof. 
\end{proof}

\begin{proof}[Proof of Proposition \ref{prophomoglinftypolyh}]
  By Theorem \ref{theo:homog} 
 there are functions $u_j\in SBV(\Omega;\sigma_j\Z^N)$ such that $u_j$ converges 
to $u$ strongly in $L^1(\Omega;\R^N)$ and
 \begin{equation*}
  \limsup_{j\to\infty} \int_{J_{u_j}\cap\Omega} \sigma_j 
\psi\left(\frac{[u_j]}{\sigma_j}, \nu_j\right) d\calH^1\le \ELT_0[u,\Omega].
 \end{equation*}
 From  $\frac1c\,|b|\le \psi(b,n)$ we obtain that $u_j$ is bounded in $BV(\Omega;\R^N)$.
 We apply
  Lemma \ref{lemma64cgo2015new} to $u_j$ with $\eta_j:=1/j$ and obtain a 
polyhedral map 
  $z_j\in SBV(\Omega;\sigma_j\Z^N)$ such that
  \begin{equation*}
   |Dz_j|(\Omega)\le 5|Du_j|(\Omega)\le C_*\,, \hskip5mm
\limsup_{j\to\infty}   \|z_j-u_j\|_{L^1(\Omega)} =0,
  \end{equation*}
for some $C_*>0$ and
\begin{equation*}
  \limsup_{j\to\infty} 
   \int_{J_{z_j}\cap\Omega} \sigma_j \psi\left(\frac{[z_j]}{\sigma_j}, 
\nu_j\right) d\calH^1\le \ELT_0[u,\Omega] .
 \end{equation*}  
Since $z_j$ is polyhedral, there are finitely many segments $[a_h,b_h]$ such 
that
\begin{equation}\label{eqdzjsegm}
 Dz_j=\sum_h (z_h^+-z_h^-) \otimes n_h \calH^1\LL[a_h,b_h],
\end{equation}
where 
for simplicity we do not indicate the 
index $j$ on the traces, the normal, and the points.
 
Let $f_j:=|z_j|$. Then $f_j\in BV(\Omega;[0,\infty))$ with $|Df_j|(\Omega)\le 
|Dz_j|(\Omega)\le C_*$. 
Possibly increasing $C_*$ we can assume 
$2^{C_*}>\|u\|_{L^\infty(\Omega)}$.
By the coarea formula,
\begin{equation*}
 |Df_j|(\Omega)=\int_0^\infty \calH^1(\partial_* \{f_j>t\}) dt
 =\sum_{k\in\Z}\int_{2^k}^{2^{k+1}} \calH^1(\partial_* \{f_j>t\}) dt\leq C_*
\end{equation*}
(we use the short notation $\partial_* \{f_j>t\}$ for $\Omega\cap 
\partial_*\{x\in \Omega: f_j(x)>t\}$, where $\partial_*$ denotes the essential 
boundary).
Fixed $\delta>0$, for any $j$ we choose $k_j\in{\N\cap} (C_*, C_*+C_*/\delta+1)$ 
 such that
\begin{equation*}
\int_{2^{k_j}}^{2^{k_j+1}} \calH^1(\partial_* \{f_j>t\}) dt\le \delta
\end{equation*}
 and then pick $M_j\in (2^{k_j},2^{k_j+1})$ such that
\begin{equation}\label{eqchoicemj}
 2^{k_j} \calH^1(\partial_* \{ f_j> M_j\})\le  \delta.
\end{equation}

We now define $\hat z_j:\Omega\to \sigma_j \Z^N$ by
\begin{equation*}
 \hat z_j(x):=\begin{cases}
 z_j(x), & \text{ if } f_j(x)\le M_j, \\
0 ,& \text{ otherwise.}
          \end{cases}
\end{equation*}
From $z_j\to u$ pointwise almost everywhere and $M_j>2^{C_*}>\|u\|_{L^\infty(\Omega)}$ we 
deduce that $\hat z_j\to u$ pointwise almost everywhere.
It is easy to check that $\hat z_j\in BV(\Omega;\sigma_j\Z^N)$ (indeed $J_{\hat z_j}\subseteq {J_{z_j}}$), and that
\begin{equation*}
 \|\hat z_j\|_{L^\infty(\Omega)} \le M_j\le C_\delta
\end{equation*}
(where $C_\delta:=2^{2+C_*+C_*/\delta}$). Further, by (\ref{eqchoicemj})
\begin{equation}\label{dhatzjdelta}
\begin{split}
 |D\hat z_j|(\Omega) \le &|D z_j|(\Omega) +  M_j \calH^1(\partial_* \{ 
f_j(x)> M_j\})\\
 \le& |D z_j|(\Omega) +  2\delta
 \le C^* +  2\delta.
\end{split}
 \end{equation}
 Therefore $\hat z_j$ converges to $u$ weakly in $BV(\Omega;\R^N)$.

 It remains to estimate the energy. The natural bound
 \begin{equation*}
 \begin{split}
 \ELT_{\sigma_j}[\hat z_j,\Omega]\le &
 \ELT_{\sigma_j}[z_j,\Omega]+\int_{\partial_*\{f_j>M_j\}} \sigma_j 
\psi\left(\frac{[\hat z_j]}{\sigma_j},\nu\right) d\calH^1 
 \end{split}
 \end{equation*}
 does not give the stated result since we do not assume linear control on $\psi$ from above (indeed, in the specific application of interest here $\psi$ is quadratic 
 in the first argument, as is apparent from (\ref{eqpsifromgamma})). 
 Therefore we need another construction, to separate big jumps into 
many small jumps, which corresponds to the fact that the relaxed energy $\psi_\rel$ has linear growth in the first argument.
We shall use that 
from the assumption $\psi(b,n)\le (1+c|n-n'|) \psi(b,n')$ for all $b\in\Z^N$ and $n,n'\in S^1$, we clearly have there exists a constant $\hat c>0$ such that
\begin{equation}\label{eqpsiupbloc}
 \psi(b,n)\le \hat c |b| \hskip5mm\text{ for all } b\in [-1,1]^N\cap\Z^N \text{ and all } n\in S^1.
\end{equation}

Recalling (\ref{eqdzjsegm}) we see that
\begin{equation}\label{eqdzjsegm2}
 D\hat z_j=\sum_h (\hat z_h^+-\hat z_h^-) \otimes n_h \calH^1\LL[a_h,b_h],
\end{equation}
where the segments are the same as in (\ref{eqdzjsegm}), and
$\hat z_h^+=z_h^+$ if $|z_h^+|\le M_j$ and $0$ otherwise, and correspondingly 
$\hat z_h^-$.
\begin{figure}
	\begin{center}
		\includegraphics[width=4cm]{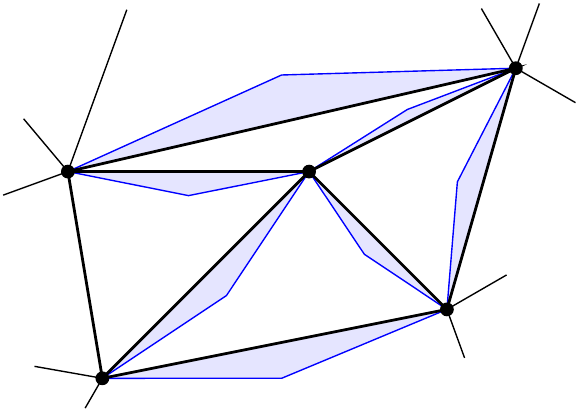}\hskip5mm
		\includegraphics[width=5.5cm]{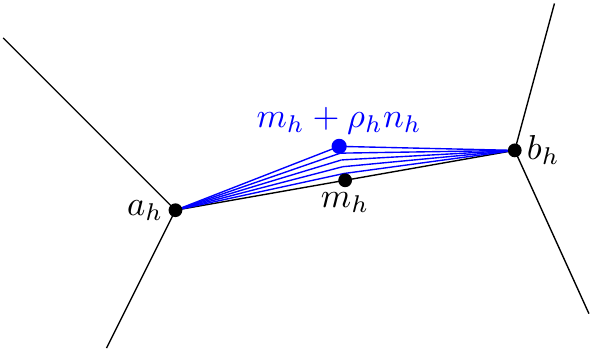}
		\caption{Sketch of the construction in Proposition \ref{prophomoglinftypolyh}. 
			Left panel: construction of the triangles around the segments on which $\hat 
			z_j$ jumps. Right panel: separation of one ``large'' jump into many smaller 
			jumps.}
		\label{figprophomoglinftypolyh}
	\end{center}
\end{figure}
The segments for which both traces are 
unchanged need not be treated, as well as those where both new traces are zero. 
The critical set is
\begin{equation*}
 H:=\{h: 0=\hat z_h^+ \ne z_h^+ \text{ and } \hat z_h^-\ne0\}\cup \{h: 0=\hat 
z_h^- \ne z_h^- \text{ and } \hat z_h^+\ne0\}.
\end{equation*}
As in the  computation in (\ref{dhatzjdelta}) we obtain
\begin{equation*}
 \sum_{h\in H} |D\hat z_j|([a_h,b_h]) 
 \le M_j \calH^1(\partial_*\{f_j>M_j\})
 \le 2\delta.
\end{equation*}
For these segments we need to separate the jump into many smaller jumps.
For any $h\in H$ we let $m_h:=(a_h+b_h)/2$ be the midpoint of the segment {$[a_h,b_h]$}, and 
choose $\rho_h\in (0, |b_h-a_h|)$ 
such that the triangles $T^h:=\conv(a_h,b_h, m_h+\rho_h n_h)$ are, up to the vertices, all disjoint and 
their total area 
is less then $2^{-j}$ (see Figure \ref{figprophomoglinftypolyh}). This is 
possible since there are finitely many segments. 

Choose now one $h\in H$, and assume for definiteness that $\hat z_h^-=0$. Since $|\hat z_h^+|\le  M_j$ there
are ${L_h}\le M_j/\sigma_j$ and 
$\alpha_l\in\sigma_j(\Z^N\cap 
[-1,1]^N)$, with $l=1, \dots, L_h$, such that $\hat z_h^+=\sum_{l=1}^{L_h} \alpha_l$. For $l=0,\dots, L_h$ we define the triangles
\begin{equation*}
 T^h_l := \conv(a_h, b_h, m_h+{\frac {l+1}{L_h+1} } \rho_h n_h)
\end{equation*}
and $v_j:\Omega\to\sigma_j\Z^N$ by setting 
$v_j=\sum_{l'=1}^l \alpha_{l'}$ on each $T^h_l\setminus T^h_{l-1}$, and $v_j=\hat z_j$ 
outside the union of the triangles. 
Then {$[v_j]\in \sigma_j(\Z^N\cap 
[-1,1]^N)$} on each of the closed triangles ${T^h}$,  $v_j=\hat z_j$ on 
the outer boundary of $T^h$, and 
$|Dv_j|( T^h)\le c |a_h-b_h| |\hat z_h^+-\hat z_h^-|\le c |D\hat 
z_j|([a_h,b_h])$. Therefore, recalling (\ref{eqpsiupbloc}),
 \begin{equation*}
 \begin{split}
 \ELT_{\sigma_j}[v_j,\Omega]\le &
 \ELT_{\sigma_j}[z_j,\Omega]+\sum_{h\in H} \int_{\cup_l \partial T_l^h} \sigma_j 
\psi\left(\frac{[v_j]}{\sigma_j},\nu\right) d\calH^1 \\
\le &
 \ELT_{\sigma_j}[ z_j,\Omega]
 +\sum_{h\in H} \int_{\cup_l \partial T_l^h} \hat c |[v_j]| d\calH^1\\
= &
 \ELT_{\sigma_j}[ z_j,\Omega]+c\sum_{h\in H} |Dv_j|( T^h)\\
\le &
 \ELT_{\sigma_j}[z_j,\Omega]+c\sum_{h\in H} |D\hat z_j|([a_h,b_h])
\le  \ELT_{\sigma_j}[ z_j,\Omega]+c\delta.
 \end{split}
 \end{equation*}
 The same computation also shows that $v_j$ is bounded in $BV$.
Since $|\{v_j\ne \hat z_j\}|\le 2^{-j}$ we obtain $v_j\to u$ pointwise almost 
everywhere.
\end{proof}

\section{Upper bound}
\label{sectupper}
The upper bound is obtained by an explicit but involved construction, that combines the {several} rescaling steps. 
Due to the incompatibility of the two constraints of being $BV$ with values in a scaled copy of $\Z^N$ and being in $H^{1/2}$ we cannot use density and separate the two scales. Instead we need to use a joint construction, which depends on both scales.

We start from the sequence constructed in Section \ref{secedenstiy}, which takes 
values in $SBV(\Omega;\sigma\Z^N)$ for a scale $\sigma$ which converges 
slowly to 0 with respect to $\eps$.
The key step is the construction in the following lemma.

\begin{lemma}\label{propupperbound}
	Let $\Omega\subset\subset\Omega'$ be two bounded Lipschitz domains.
	Let $\sigma>0$, $v\in SBV(\Omega';\sigma \Z^N)
	{\cap L^\infty(\Omega';\sigma \Z^N)}$ polyhedral,
	$\alpha\in(0,{\frac12})$ and $\rho>0$ with $ 3\rho^\alpha<\dist(\Omega,\partial\Omega')$.
	
	Then for any $\eps>0$ there are $w_\eps\in L^2(\Omega;\R^N)$ and $\zeta\in \overline B_1$ such that
	\begin{equation*}
	\|\frac{w_\eps}{\ln(1/\eps)}\|_{L^\infty(\Omega)}\le 
	\|v\|_{L^\infty(\Omega)},
	\end{equation*}
	\begin{equation*}
	\|\frac{w_\eps}{\ln(1/\eps)}- v\|_{L^1(\Omega)}\le \rho^\alpha |Dv|(\Omega'),
	\end{equation*}
	and
	\begin{align*}
	\limsup_{\eps\to0} \frac{ E_\eps[w_\eps,\Omega]}{(\ln(1/\eps))^2} &\le \int_{J_v\cap \Omega'} \sigma \psi\left(\frac{[v]}{\sigma},n\right) d\calH^1 
	+ f(\rho) (|Dv|(\Omega'))^{4/3} {\|v\|_{L^\infty(\Omega')}^{2/3}}\\ 
	+ \int_{\Omega} \int_{\Omega\setminus B_\rho(x)}& \Gamma(x-y)(v^\zeta_\infty(x)-v^\zeta_\infty(y))\cdot (v^\zeta_\infty(x)-v^\zeta_\infty(x)) dydx 
	\end{align*}
	where 
	$v^\zeta_\infty: \Omega\to\R^N$ is defined by
	\begin{equation}\label{inftyaverage}	
	v^\zeta_\infty (x):= \int_0^1 v(x+\rho^\alpha \zeta t) dt,
	\end{equation}
	and 
	$f(\rho)\to0$ as $\rho\to0$. The function $f$ depends on $\Gamma$ and $\alpha$, but not on $v$.
\end{lemma}

Proof will be given at the end of this section. The main point is to replicate each interface 
$\sigma\ln\frac1\eps$ times, and then mollify on a scale $\eps$. This modifies the function only on a small set, of area proportional to $\sigma \eps \ln \frac1\eps$, which ensures that the nonlinear term $\eps^{-1}\|\dist(u_\eps,\Z^N)\|_2^2$ vanishes in the limit, for an appropriate scaling of $\sigma$. Since the separation between the interfaces is much larger than the scale of the mollification, their interaction is small. For each interface, the energy is estimated by an explicit computation in Lemma \ref{lemmamollificationupper}. 

Care must be taken in undoing the several relaxation steps, both at the line-tension and at the continuous scale, and in several truncation steps to permit to estimate the various error terms.
For this construction in the upper bound we fix a mollifier $\varphi_1\in C^\infty_c(B_1)$ and set $\varphi_\lambda:=\lambda^{-2}\varphi_1(\lambda x)$.

\begin{proposition}\label{propgammalimsup}
 Let $\Omega\subset\R^2$ be a bounded connected Lipschitz domain,
and let $E_\eps[\cdot, \Omega]$ be defined as in (\ref{uno}), with 
$W$ and $\Gamma$ which satisfy 
(\ref{eqWdist})--(\ref{eqgammaposdeftheointro}).

Let $u_0\in BV(\Omega;\R^N)\cap H^{1/2}(\Omega;\R^N)$. For any  $\eps>0$ there is $u_\eps\in L^2(\Omega;\R^N)$ with
 $u_\eps/(\ln(1/\eps))\to u$ in $L^2$ and
 \begin{equation}\label{eqpropgammalimsup}
     \limsup_{\eps\to0} \frac{E_\eps[u_\eps,\Omega]}{(\ln(1/\eps))^2}\le 
F_0[u,\Omega]\,.
 \end{equation}
\end{proposition}
{We recall that $F_0$ was defined in (\ref{eqdefF}).}
\begin{proof}
 We start by reducing to the case that $u$ is smooth and that it is defined on a domain $\Omega'$ larger  than $\Omega$.

 To see this, observe that since $\Omega$ is Lipschitz there are an open set $\omega$ with $\partial\Omega\subset\omega$ and a 
  bilipschitz map $\Phi:\omega\to\omega$ such that $\Phi(x)=x$ for $x\in\partial\Omega$ and
 $\Phi(\Omega\cap \omega)=\omega\setminus\overline\Omega$. We define $\hat u:\Omega\cup\omega\to\R^N$ by reflection
 \begin{equation*}
  \hat u := \begin{cases}
            u ,& \text{ in } \Omega,\\
            u \circ \Phi, &\text{ in } \omega\setminus\Omega.
           \end{cases}
 \end{equation*}
 Then $\hat u\in BV(\Omega\cup\omega;\R^N)\cap H^{1/2}(\Omega\cup\omega;\R^N)$, with $|D\hat 
u|(\partial\Omega)=0$.
We fix $\delta>0$ and let $\Omega_\delta:=\{x:\dist(x,\Omega)<\delta\}$, so that
 \begin{equation*}
  \limsup_{\delta\to0} |D\hat u|(\Omega_\delta) =|D\hat u|(\Omega)  =|Du|(\Omega) 
 \end{equation*}
and
\begin{equation*}
  \limsup_{\delta\to0}\, [\hat u]_{H^{1/2}(\Omega_\delta)} =[\hat 
u]_{H^{1/2}(\Omega)}=[u]_{H^{1/2}(\Omega)}.
 \end{equation*}
In particular, 
\begin{equation*}
  \limsup_{\delta\to0} F_0[\hat u,\Omega_\delta] =F_0[\hat u,\Omega] 
=F_0[u,\Omega].
 \end{equation*}
Now for sufficiently small $\delta$ define
$u_\delta:=\varphi_\delta\ast \hat u\in C^\infty(\overline{\Omega_\delta};\R^N)$, with $\varphi_\delta$ the  mollification kernel. Since $F_0$ is convex, we have
\begin{equation*}
  \limsup_{\delta\to0} F_0[u_\delta,\Omega_\delta] \le 
  \limsup_{\delta\to0} F_0[\hat u,\Omega_{2\delta}] =F_0[ u,\Omega]
 \end{equation*} 
 and $u_\delta\to u$ in $L^2(\Omega;\R^N)$.
 
 Therefore in the rest of the proof we assume that $u\in C^\infty(\overline{\Omega'};\R^N)$ is given, 
with $\Omega\subset\subset\Omega'$ and $\Omega'$ Lipschitz. We shall show that for any
 $\eta>0$ and any $\eps>0$ there is $w_\eps\in L^2(\Omega;\R^N)$ such that $w_\eps/\ln(1/\eps)\to u$ in 
$L^2(\Omega;\R^N)$ and
 \begin{equation}\label{eqendofproof}
  \limsup_{\eps\to0} \frac{1}{(\ln (1/\eps))^2} E_\eps[w_\eps,\Omega] \le F_0[u,\Omega'] +\eta.
 \end{equation}
 Since $\eta$ is arbitrary, taking a diagonal subsequence will conclude the 
proof of (\ref{eqpropgammalimsup}). 

It remains to prove (\ref{eqendofproof}).  Let $\sigma_j\in (0,1)$ be such that 
$\sigma_j\downarrow 0$ as $j\to0$.  
 By Proposition \ref{prophomoglinftypolyh} (which can be applied thanks to Lemma \ref{lemmapsilipschit})
 there are polyhedral functions  $v_j\in SBV(\Omega';\sigma_j\Z^N)$ such that 
$v_j$ converges to $u$ strongly in $L^1(\Omega';\R^N)$,
 \begin{equation*}
  \limsup_{j\to\infty} \int_{J_{v_j}\cap\Omega'} \sigma_j \psi\left(\frac{[v_j]}{\sigma_j}, n_j\right) d\calH^1\le F_\self[u,\Omega'] +\eta,
 \end{equation*}
 and $C_\eta:=\sup\|v_j\|_{L^\infty(\Omega)}+|Dv_j|(\Omega')<\infty$. We recall that, {since $u$ is smooth,} in particular $u\in L^\infty(\Omega')$
{and that $F_\self[u,\Omega']=\ELT_0[u,\Omega']$ (see the statements of Theorem \ref{theorem1} and Theorem \ref{theo:homog}).}
 
Since $v_j$ is polyhedral, by Lemma \ref{propupperbound} for $\alpha:=1/3$, and  $\eps$ and $\rho$ small enough, there are 
functions $w_\eps^{j,\rho}\in L^2(\Omega;\R^N)$ and vectors $\zeta_j\in \overline{B}_1$ such that
\begin{equation}\label{eqconstruct1-w}
\|\frac{w_\eps^{j,\rho}}{\ln(1/\eps)}\|_{L^\infty(\Omega)}\le 
\|v_j\|_{L^\infty(\Omega)}\le C_\eta,
\end{equation}
	\begin{equation*}
	\|\frac{w_\eps}{\ln(1/\eps)}- v\|_{L^1(\Omega)}\le \rho^\alpha |Dv|(\Omega'),
	\end{equation*}
and
\begin{align*}
 \limsup_{\eps\to0} &\frac{1}{(\ln(1/\eps))^2} E_\eps[w_\eps^{j,\rho},\Omega] \\
 \le& \int_{J_{v_j}\cap\Omega'} \sigma_j \psi\left(\frac{[v_j]}{\sigma_j},n_j\right) d\calH^1 
+ f(\rho) (|Dv_j|(\Omega'))^{4/3} \|v_j\|_{L^\infty(\Omega')}^{2/3}\\ 
 &+ \int_{\Omega} \int_{\Omega\setminus B_\rho(x)} \Gamma(x-y)(v^{\zeta_j}_{j,\infty}(x)-v^{\zeta_j}_{j,\infty}(y))\cdot (v^{\zeta_j}_{j,\infty}(x)-v^{\zeta_j}_{j,\infty}(y)) dydx ,
\end{align*}
where $v^{\zeta_j}_{j,\infty}=(v_j)_\infty^{\zeta_j}$ is an average of $v_j$ in direction $\zeta_j$ at a scale set by $\rho$ as defined in \eqref{inftyaverage}.
Further from \eqref{eqconstruct1-w}
\begin{align*}
 \limsup_{\eps\to0}  \|\frac{w_\eps^{j,\rho}}{\ln(1/\eps)}-u\|_{L^1(\Omega)}\le&
 \|v_j-u\|_{L^1(\Omega)}+
\limsup_{\eps\to0}  \|\frac{w_\eps^{j,\rho}}{\ln(1/\eps)}-v_j\|_{L^1(\Omega)}\\
\le&  \|v_j-u\|_{L^1(\Omega)}+
 \rho^\alpha |Dv_j|(\Omega') \\
\le & \|v_j-u\|_{L^1(\Omega)}+ C_\eta \rho^\alpha.
\end{align*}
We now take $j\to\infty$, and extract a subsequence such that $\zeta_{j_k}\to\zeta_\infty$
{and $\displaystyle\lim_{k\to\infty} \limsup_{\eps\to0} \frac{1}{(\ln(1/\eps))^2} E_\eps[w_\eps^{j_k,\rho},\Omega] =
\limsup_{j\to\infty} \limsup_{\eps\to0} \frac{1}{(\ln(1/\eps))^2} E_\eps[w_\eps^{j,\rho},\Omega]$.}
By dominated convergence,
$(v_j)^{\zeta_j}_\infty\to u^{\zeta_\infty}_\infty$ pointwise and hence in  $L^2(\Omega;\R^N)$ 
 and 
\begin{align*}
\limsup_{j\to\infty} \limsup_{\eps\to0} \frac{1}{(\ln(1/\eps))^2} E_\eps[w_\eps^{j,\rho},\Omega] &\le F_\self[u,\Omega']+\eta
+ C_\eta^ 2 f(\rho) \\ 
 + \int_{\Omega} \int_{\Omega\setminus B_\rho(x)} \Gamma(x-y)&(u^{\zeta_\infty}_{\infty}(x)-u^{\zeta_\infty}_{\infty}(y))\cdot (u^{\zeta_\infty}_{\infty}(x)-u^{\zeta_\infty}_{\infty}(y)) dydx. 
\end{align*}
As $\rho\to0$ we have, since $u\in C^\infty(\Omega';\R^N)$, that
 $u^{\zeta_\infty}_{\infty}\to u$ in $L^2(\Omega;\R^N)$ and in $H^{1/2}(\Omega;\R^N)$ and therefore
\begin{align*}
\limsup_{\rho\to0}\limsup_{j\to\infty} \limsup_{\eps\to0} \frac{1}{(\ln(1/\eps))^2} E_\eps[w_\eps^{j,\rho},\Omega] &\le F_\self[u,\Omega']+\eta
\\ 
 + \int_{\Omega'\times\Omega'} \Gamma(x-y)&(u(x)-u(y))\cdot(u(x)-u(y)) dydx \\
 &= F_0[u,\Omega']+\eta
\end{align*}
with 
\begin{equation*}
 \limsup_{\rho\to0}\limsup_{j\to\infty} \limsup_{\eps\to0}  \|\frac{w_\eps^{j,\rho}}{\ln(1/\eps)}-u\|_{L^2(\Omega)}=0.
\end{equation*}

Taking a diagonal sequence concludes the proof of (\ref{eqendofproof}).
 \end{proof}

It remains to show the detailed construction of the functions $w_\eps^{j,\rho}$ given in Lemma \ref{propupperbound}. First we recall that the unrelaxed line tension energy for polyhedral interfaces can be obtained with a direct computation starting from the nonlocal energy.  
 \begin{lemma}\label{lemmamollificationupper}
Let  $\Omega\subset\subset\Omega'$ be two bounded open sets, $v\in SBV(\Omega'; \Z^N)$ polyhedral,
and assume that $\Gamma$ obeys
(\ref{eq:derGammaintro}) and (\ref{eqgammaposdeftheointro}). Let  $\varphi_\eps\in C^\infty_c(B_\eps)$  be a mollifier and 
$\eps<\dist(\Omega,\partial\Omega')$. Then $w_\eps:=\varphi_\eps\ast v$ obeys
\begin{align*}
 &\limsup_{\eps\to0} \frac{1}{\ln(1/\eps)}\int_{\Omega\times \Omega} \Gamma(x-y) (w_\eps(x)-w_\eps(y))\cdot(w_\eps(x)-w_\eps(y)) dxdy\\
 &
 \le \int_{J_v\cap\Omega'} \psi([v],n) d\calH^1,
\end{align*}
where $\psi$ is as in (\ref{eqpsifromgamma}).
\end{lemma}
\begin{proof}
See \cite[Section 6]{GarroniMueller2006}.
\end{proof}

\begin{proof}[Proof of Lemma \ref{propupperbound}]

We choose $\Omega''$ such that $\Omega\subset\subset\Omega''\subset\subset\Omega'$, 
{ $\rho^\alpha<\dist(\Omega,\partial\Omega'')$,}
and
 $\rho^\alpha<\dist(\Omega'',\partial\Omega')$, and for
$\zeta\in\overline B_1\subset\R^2$ and $L>0$ we define the functions
$v^\zeta_L: \Omega''\to\frac{\sigma}{L}\Z^N$ by
\begin{equation*}
 v^\zeta_L (x):=\sum_{j=1}^{\lfloor L\rfloor} \frac1L v(x+\rho^\alpha\frac jL\zeta).
\end{equation*}
Note that $v^\zeta_L\in SBV(\Omega'';\frac{\sigma}{L}\Z^N)$ has jump set which is  obtained by $\lfloor L\rfloor$ copies of the jump set of $v$, translated in the direction of $\zeta$, and that $\|v^\zeta_L\|_{L^\infty(\Omega'')}\leq \|v\|_{L^\infty(\Omega')}$.
	 
We set $L_\eps:= \sigma \ln \frac1\eps$.
For $\eps\le \dist(\Omega,\partial\Omega'')$ 
we define 
\begin{equation*}
w_\eps:={\ln (1/\eps)} v^{\zeta_\eps}_{L_\eps}\ast \varphi_\eps \hskip5mm\text{ and }\hskip5mm
 \hat w_\eps:= v^{\zeta_\eps}_{L_\eps}\ast \varphi_\eps
\end{equation*}
(if $\eps>\dist(\Omega,\partial\Omega'')$ we can set $w_\eps=0$),
the vectors $\zeta_\eps\in\overline B_1$ will be chosen below.

We remark that $v\in \sigma\Z^N$ almost everywhere implies $v_{L_\eps}^{\zeta_\eps} \in \frac{1}{L_\eps} \sigma \Z^N=\frac{1}{\ln(1/\eps)} \Z^N$ a.e., therefore
 $w_\eps(x)\in\Z^N$ for any $x$ at distance at least $\eps$ from $J_{v^{\zeta_\eps}_{L_\eps}}$.
Since $J_v$ is a finite union of segments, 
{and $\dist(w_\eps,\Z^N)\le N^{1/2}$ everywhere,}
we have
 \begin{equation*}
  \limsup_{\eps\to0}\frac{1}{\eps (\ln(1/\eps))^2} \int_{\Omega} \dist^2(w_\eps,\Z^N) dx 
  \le \lim_{\eps\to0}\frac{\lfloor L_\eps\rfloor\,  N\, |\{x\in \Omega: \dist(x, 
J_v\cap\Omega')<\eps\}|}{\eps  (\ln(1/\eps))^2} =0.
 \end{equation*}

 \newcommand\epsi{{\eps_i}}
 We write the long-range elastic energy as a bilinear form, $B_{LR}^\rho:L^2(\Omega;\R^N)^2\to\R$,
\begin{equation*}
 B_{LR}^\rho(u,u'):=\int_{\Omega}\int_{\Omega\setminus B_\rho(x)} \Gamma(x-y) (u(x)-u(y))\cdot (u'(x)-u'(y)) dy dx,
\end{equation*}
and choose a sequence $\eps_i\to0$, $\eps_i>0$, such that 
\begin{equation}\label{eqblrepspespi}
\lim_{i\to\infty} B_{LR}^\rho(\hat w_{\eps_i},\hat w_{\eps_i})
=\limsup_{\eps\to0} B_{LR}^\rho(\hat w_\eps,\hat w_\eps).
\end{equation}
After extracting a further subsequence, we can additionally assume that $\zeta_{\eps_i}\to\zeta$ for some $\zeta\in \overline B_1$.
This defines the vector $\zeta$ in the statement (in terms of the vectors $\zeta_\eps$ chosen below).
We now show that 
\begin{equation}\label{eqwepsivzetain}
\hat w_{\eps_i}\to v^\zeta_\infty \text{ strongly in $L^2(\Omega;\R^N)$.}
\end{equation}
To see this, we first observe that since
$\|\hat w_\epsi\|_{L^\infty(\Omega)}\le \|v\|_{L^\infty(\Omega')}$
and 
$\|v_\infty^\zeta\|_{L^\infty(\Omega)}\le \|v\|_{L^\infty(\Omega')}$
it suffices to prove convergence in $L^1(\Omega;\R^N)$.
We write
\begin{align*}
 \hat w_{\eps_i}-v^\zeta_\infty=
 \varphi_\epsi\ast (v_{L_\epsi}^{\zeta_\epsi}-v_\infty^{\zeta_\epsi}) +
 \varphi_\epsi\ast(v_\infty^{\zeta_\epsi}-v_\infty^{\zeta})
 + (\varphi_\epsi \ast v_\infty^{\zeta}-v_\infty^{\zeta})
\end{align*}
and estimate the three terms separately. Convergence of the last one is immediate.
Performing an explicit computation one can show that
\begin{align*}
 \|v^{\zeta_{\eps_i}}_\infty-v^\zeta_\infty\|_{L^1(\Omega'')} &\le \rho^\alpha |Dv|(\Omega') |\zeta_{\eps_i}-\zeta|,
\end{align*}
which implies $\|\varphi_\epsi\ast(v_\infty^{\zeta_\epsi}-v_\infty^{\zeta})\|_{L^1(\Omega)}
\le \|v_\infty^{\zeta_\epsi}-v_\infty^{\zeta}\|_{L^1(\Omega'')}\to0$.
Analogously, from
\begin{align*}
 \|v^{\zeta_{\eps_i}}_{L_\epsi}-v^{\zeta_{\eps_i}}_\infty\|_{L^1(\Omega'')} &\le \frac{\rho^\alpha}{L_\epsi} |Dv|(\Omega')
+|\Omega''| \frac{L_\eps-\lfloor L_\eps\rfloor}{L_\eps} \|v\|_{L^\infty(\Omega'')}
 \end{align*}
and $\lim_{i_\to\infty} L_\epsi=\infty$
we obtain
$\| \varphi_\epsi\ast (v_{L_\epsi}^{\zeta_\epsi}-v_\infty^{\zeta_\epsi})\|_{L^1(\Omega)}\to0$. This concludes the proof of (\ref{eqwepsivzetain}).

By continuity of $B_{LR}^\rho$, (\ref{eqblrepspespi}) and (\ref{eqwepsivzetain}) imply
\begin{equation*}
\limsup_{\eps\to0} B_{LR}^\rho(\hat w_\eps,\hat w_\eps)
=\lim_{i\to\infty} B_{LR}^\rho(\hat w_{\eps_i},\hat w_{\eps_i})=
B_{LR}^\rho(v^\zeta_\infty,v^\zeta_\infty).
\end{equation*}
The short-range elastic energy can be correspondingly written, for 
a Borel set $\Borel\subseteq\R^2$,
as the bilinear form $B_{SR}^\rho(\cdot,\cdot,\Borel):L^2(\Borel;\R^N)^2\to\R$,
\begin{equation*}
 B_{SR}^\rho(u,u',\Borel):=\int_{\Borel}\int_{\Borel\cap B_\rho(x)} \Gamma(x-y) (u(x)-u(y))\cdot (u'(x)-u'(y)) dy dx.
\end{equation*}
This term will lead us to the choice of $\zeta_\eps$. 
We are interested in showing that for any $\eps$ there is a choice of $\zeta\in \overline{B_1}$ which permits to control the quantity
\begin{equation*}
 B_{SR}^\rho(\hat w_\eps, \hat w_\eps,\Omega)=\frac{1}{L_\eps^2} \sum_{j,j'=1}^{\lfloor L_\eps\rfloor} B_{SR}^\rho(T_j^\zeta v\ast \varphi_\eps, T_{j'}^\zeta v\ast \varphi_\eps,\Omega)
\end{equation*}
where $T_j^\zeta$ is the translation operator, $(T_j^\zeta f)(x):=f(x+j\zeta\rho^\alpha/L_\eps)$.
The separation introduced by the translations is on a length 
scale much larger than $\eps$, but still infinitesimal
(the choice of $\zeta_\eps$ below shall implicitly ensure that it is not too small), therefore it is appropriate to treat the diagonal ($j=j'$) terms separately.
Using translation invariance we can see that the diagonal contribution is
\begin{align*}
 B_{SR}^{\mathrm{diag}}(\zeta)&:=\frac{1}{L_\eps^2} \sum_{j=1}^{\lfloor L_\eps\rfloor} B_{SR}^\rho(T_j^\zeta v\ast \varphi_\eps, T_{j}^\zeta v\ast \varphi_\eps,\Omega)\\
&\le   \frac{\lfloor L_\eps\rfloor}{L_\eps^2} B_{SR}^\rho(v\ast \varphi_\eps,v\ast \varphi_\eps,\Omega'')\\
 &=  \frac{\lfloor L_\eps\rfloor \sigma}{L_\eps \ln(1/\eps)} B_{SR}^\rho(\sigma^{-1}v\ast \varphi_\eps,  \sigma^{-1}v\ast \varphi_\eps,\Omega'')
\end{align*}
and  in particular that {the latter expression} does not depend on the choice of $\zeta$. 
Since $\sigma^{-1} v\in SBV({\Omega'};\Z^N)$ is polyhedral and $\Gamma\ge0$ pointwise,
recalling Lemma \ref{lemmamollificationupper}, we obtain
\begin{equation*}
 \limsup_{\eps\to0}B_{SR}^{\mathrm{diag}}(\zeta)\le \int_{J_v\cap\Omega'} \sigma \psi\left(\frac{[v]}\sigma,n\right) d\calH^1
\end{equation*}
for any $\zeta\in \overline{B_1}$.

The off-diagonal contributions reduce to
\begin{align*}
B_{SR}^\mathrm{cross}(\zeta):=&\frac{1}{L_\eps^2} \sum_{j\ne j'} B_{SR}^\rho( 
T^\zeta_{j} v\ast \varphi_\eps, T^\zeta_{j'} v\ast \varphi_\eps,\Omega)\\
\le& \frac{1}{L_\eps^2} \sum_{j\ne j'} B_{SR}^\rho( v\ast \varphi_\eps, T^\zeta_{j-j'} v\ast \varphi_\eps,\Omega'').
\end{align*}
We average over all possible choices of the shifts $\zeta$. Precisely, we 
compute, using linearity of $B_{SR}^\rho$ in the second argument,
\begin{align*}
A_\eps:=& \int_{B_1} B_{SR}^\mathrm{cross}(\zeta) \varphi_{1}(\zeta) d\zeta 
\le\int_{B_1} \frac{1}{L_\eps^2} \sum_{j\ne j'} B_{SR}^\rho( v\ast \varphi_\eps, 
\varphi_{1}(\zeta) T_{j-j'}^\zeta v\ast \varphi_\eps,\Omega'')d\zeta\\
=&B_{SR}^\rho( v\ast \varphi_\eps, V\ast \varphi_\eps,\Omega''),
\end{align*}
with
\begin{equation*}
 V(x):=\int_{B_1} \frac{1}{L_\eps^2} \sum_{j\ne j'}
\varphi_{1}(\zeta) v(x+\frac{j-j'}{L_\eps} \rho^\alpha \zeta) d\zeta.
\end{equation*}
By a change of variables we obtain $V=  \Phi\ast v$, 
where
\begin{equation*}
 \Phi(x):= \frac{1}{L_\eps^2} \sum_{j\ne j'} \frac{L_\eps^2}{(j-j')^2} \varphi_{\rho^\alpha}\left(\frac{L_\eps}{j-j'}x\right), 
\end{equation*}
and then
\begin{equation*}
 A_\eps
\leq B_{SR}^\rho(v\ast \varphi_\eps, 
\Phi\ast v\ast \varphi_\eps,\Omega'').
\end{equation*}
We fix $p\in(2,\infty)$ and denote by $q:=p/(p-1)$ the dual exponent. Then
\begin{align*}
A_\eps\le& \int_{B_\rho} \frac{c}{|z|^3} \int_{\Omega''} | (\varphi_\eps\ast v)(x)-(\varphi_\eps\ast v)(x+z)| \\
 &\hskip1cm\,\times | (\varphi_\eps\ast \Phi\ast v)(x)-(\varphi_\eps\ast \Phi\ast v)(x+z)|dxdz \\
 &\le  \int_{B_\rho} \frac{c}{|z|^{2-1/p}}  \frac{\|v(\cdot)-v(\cdot+z)\|_{L^p(\Omega''_\eps)}}{|z|^{1/p}}  
 \frac{ \|(\Phi\ast v)(\cdot) - (\Phi\ast v)(\cdot+z)\|_{L^q(\Omega''_\eps)}}{|z|} dz,
\end{align*}
with $\Omega''_\eps:=\{x: \dist(x,\Omega'')<\eps\}$.
We estimate, for small $z$,
\begin{align*}
 \frac{ \| v(\cdot)-v(\cdot+z)\|_{L^p(\Omega''_\eps)}^p}{|z|}  
\le & 2^{p-1}\|v\|_{L^\infty(\Omega')}^{p-1}\int_{\Omega''_\eps}\frac{| v(x)-v(x+z)|}{|z|} dx
\\ \le &2^{p-1} \|v\|_{L^\infty(\Omega')}^{p-1} |Dv|(\Omega')
\end{align*}
and
\begin{equation*}
 \frac{ \| (\Phi\ast v)(\cdot) - (\Phi\ast v)(\cdot+z)\|_{L^q(\Omega''_\eps)}}{|z|} 
 \le \|\Phi\|_{L^q(\R^2)} |Dv|(\Omega'), 
\end{equation*}
so that, with $ \int_{B_\rho} \frac{1}{|z|^{2-1/p}}dz \le c \rho^{1/p}$, we conclude
\begin{equation*}
 A_\eps\le c\rho^{1/p} \|v\|_{L^\infty(\Omega')}^{1/q} 
(|Dv|(\Omega'))^{1+1/p}\|\Phi\|_{L^q(\R^2)}. 
\end{equation*}
Finally, recalling that $p>2$,
\begin{align*}
 \|\Phi\|_{L^q(\R^2)}\le&2 L_\eps \sum_{j=1}^{\lfloor L_\eps\rfloor } \frac{1}{j^2} \| 
\varphi_{\rho^\alpha} (\frac{L_\eps}{j} x) \|_{L^q(\R^2)} \\
 \le &c L_\eps \sum_{j=1}^{\lfloor L_\eps\rfloor }\frac{1}{j^2}  \left(\frac{j}{L_\eps}\right)^{2/q} \rho^{-2\alpha \frac{q-1}{q}}\\
 \le &c L_\eps^{1-2/q}  \sum_{j=1}^{\lfloor L_\eps\rfloor } \frac{1}{j^{2/p}} \rho^{-2\alpha/p}\\
\le &c L_\eps^{1-2/q} {L_\eps^{1-2/p}} \rho^{-2\alpha/p}= c \rho^{-2\alpha/p}\,.
\end{align*}
Therefore $A_\eps\le c\rho^{(1-2\alpha)/p} \|v\|_{L^\infty(\Omega')}^{1/q} (|Dv|(\Omega'))^{1+1/p}$. We finally choose $p=3$, and $\zeta_\eps$
so that it is as good as on average, in the sense that 
$B_{SR}^\mathrm{cross}(\zeta_\eps)\le A_\eps$, and conclude the proof.
\end{proof}

\section{Lower bound}
\label{sectlowerb}
In this Section we prove the lower bound. 
The idea is that the limit is given by two terms, that arise from short-range and long-range contributions to the nonlocal interaction, respectively. Indeed, one key idea in the proof of Proposition \ref{proplowerbound} is to localize the limiting energy and view it as a measure on $\Omega\times\Omega\subset\R^4$. One then shows that this measure can be written as the sum of two mutually singular terms, one supported on the diagonal and one supported outside the diagonal (see Figure \ref{figureDeltaBj}). The lower bound arises from estimating separately these two terms. In the estimate of the diagonal term, which is local in the limit, we build upon techniques obtained for a different scaling in \cite{ContiGarroniMueller2011}, see Proposition \ref{proplowerboundlocal} below. One important step is to iteratively mollify the functions along the sequence and to show that on most scales the mollification does not reduce significantly the $BV$ norm, which implies that the functions are approximately one-dimensional at that scale. 
The proof is done by showing that one can choose a scale that contains, up to higher-order terms, as much energy as the average scale, and that at the same time has a small loss of $BV$ norm, see (\ref{eqhgoodp}) and (\ref{eqhgoodpdue}) below.
\begin{proposition}\label{proplowerbound}
 Under the assumptions of Theorem \ref{theorem1}, for any 
 {$u\in BV(\Omega;\R^N)$ and any}
 sequences $\eps_i\to0$ and $u_i\in L^2(\Omega;\R^N)$ with
 $u_i/(\ln(1/\eps_i))\to u$ in $L^2(\Omega;\R^N)$ 
 one has
 \begin{equation}\label{eqproplblb}
  F_0[u,\Omega]\le  \liminf_{i\to\infty} \frac{E_{\eps_i}[u_i,\Omega]}{(\ln(1/\eps_i))^2}\,,
 \end{equation}
 where $F_0$ was defined in (\ref{eqdefF}).
\end{proposition}

The proof is based on the following local lower bound, which relates the short-range part of the energy to $F_\self$.
\begin{proposition}\label{proplowerboundlocal}
 Under the assumptions of Theorem \ref{theorem1}, 
 for any   {$u\in BV(\Omega;\R^N)$ and any}
sequences $\eps_i\to0$ and $u_i\in L^2(\Omega;\R^N)$ with
 $u_i/(\ln(1/\eps_i))\to u$ in $L^2(\Omega;\R^N)$, and any open set $\omega\subset\Omega$, 
 one has
 \begin{equation*}
  F_\self[u,\omega]\le  \liminf_{i\to\infty} \frac{E_{\eps_i}[u_i,\omega]}{(\ln(1/\eps_i))^2}\,,
 \end{equation*}
 where $F_\self$ was defined in (\ref{eqdefFself}).
\end{proposition}
We postpone the proof of Proposition \ref{proplowerboundlocal}, and first show that it implies 
Proposition \ref{proplowerbound}.
\begin{proof}[Proof of Proposition \ref{proplowerbound}]
We can assume that the $\liminf$ in (\ref{eqproplblb}) is finite and, after 
passing to a subsequence, that it is a limit.
By Proposition \ref{propcompactness} we can assume $u\in BV(\Omega;\R^N)\cap H^{1/2}(\Omega;\R^N)$.

We start by localizing the energy. We denote by $\Delta:=\{(x,x): x\in\R^2\}\subset\R^4$ the diagonal set in $\R^4$ and by $P:\R^4\to\R^2$, $P(x_1,x_2,y_1,y_2):=(x_1,x_2)$, the projection on the first two components. For any Borel set $\Borel \subseteq\Omega\times\Omega$ we define
\begin{align*}
 \mu_i(\Borel ):=&\frac{1}{\eps_i (\ln(1/\eps_i))^2} \int_{P(\Borel \cap \Delta)} \dist^2(u_i(x), \Z^N) d\calL^2(x) \\
 &+
 \int_\Borel  \Gamma(x-y) \frac{u_i(x)-u_i(y)}{\ln(1/\eps_i)} \cdot \frac{u_i(x)-u_i(y)}{\ln(1/\eps_i)} d\calL^{4}(x,y),
\end{align*}
so that $(\ln(1/\eps_i))^{-2}E_{\eps_i}[u_i,\Omega]=\mu_i(\Omega\times\Omega)$.
We observe that $\mu_i$ is a Radon measure, and after extracting a further subsequence we can assume that 
{$\mu_i$ converges weakly in measures to some measure $\mu$, which implies}
$\mu(A\times A)\le \lim_{i\to\infty} \mu_i(A\times A)$ for any open set $A\subseteq\Omega$. To conclude it suffices to prove that
\begin{equation}\label{eqf0muomegaom}
  F_0[u,\Omega]\le \mu(\Omega\times\Omega).
\end{equation}

In order to treat the long-range part of the interaction we define a measure 
$\mu_{LR}$ on $\Omega\times\Omega$ by
 \begin{equation*}
  \mu_{LR}(\Borel ) := \int_\Borel  \Gamma(x-y) (u(x)-u(y))\cdot (u(x)-u(y)) d\calL^4(x,y),
 \end{equation*}
 for any Borel set $\Borel \subseteq\Omega\times\Omega$. 
 Since $u\in H^{1/2}(\Omega;\R^N)$, we have $\mu_{LR}(\Omega{\times\Omega})<\infty$.
 Since $\Gamma(x-y) 
\xi\cdot\xi\ge0$ for any {$\xi\in\R^N$ and} $(x,y)\in\Omega\times\Omega$, and (possibly extracting a further subsequence)
 $(u_i(x)-u_i(y))/\ln(1/\eps_i)$ converges pointwise to $u(x)-u(y)$ for $\calL^4$-almost every $(x,y)$, by Fatou's Lemma we obtain
 \begin{equation*}
  \mu_{LR}(\Borel )\le  \liminf_{i\to\infty} \mu_i(\Borel )
 \end{equation*}
 for any Borel set $\Borel\subseteq\Omega\times\Omega$ 
 and in particular $\mu_{LR}(B_r^{(4)}(x))\le \mu(B_R^{(4)}(x))$ if $r<R$ and $B_R^{(4)}(x)\subset\Omega\times\Omega$, 
where $B_r^{(4)}(x)$ is the four-dimensional ball of radius $r$ centered at $x\in\R^4$.
 Since $\mu_{LR}$ is absolutely continuous with respect to $\calL^4$, we conclude
 \begin{equation}\label{eqLRMmu}
 \mu_{LR}(\Borel )\le \mu(\Borel ) \text{ for any Borel set $\Borel \subseteq\Omega\times\Omega$}.
 \end{equation}

\begin{figure}
 \begin{center}
  \includegraphics[width=5cm]{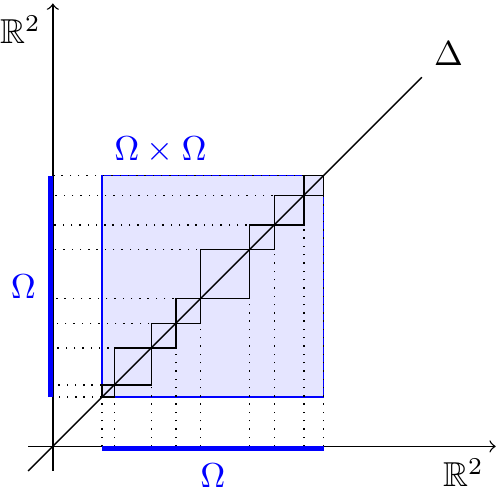}
 \end{center}
\caption{\label{figureDeltaBj} Sketch of the set $(\Omega\times \Omega)\cap\Delta$. The set $\Omega$ is covered (up to a null set) by finitely many balls $B_j$, the set $(\Omega\times \Omega)\cap\Delta$ is correspondingly covered (up to a null set) by 
the products $B_j\times B_j$.}
\end{figure}

We now deal with the short-range part of the energy, which concentrates on the diagonal set. 
We define the measure
\begin{equation*}
 \lambda := g\left(\frac{dDu}{d|Du|}\right) |Du|
\end{equation*}
so that $F_\self[u,E]=\lambda(E)$ for any Borel set $E\subseteq\Omega$ (we recall that $F_\self$ has been defined in (\ref{eqdefFself})). Since $u\in BV(\Omega;\R^N)$, $\lambda(\Omega)<\infty$.
Let $\eta>0$. For each $x\in \Omega$ there are arbitrarily small $r>0$ such that
$B_{2r}(x)\subset\Omega$, $2r<\eta$,  $\mu(\partial (B_r(x)\times B_r(x)))=0$ and 
$\lambda(\partial B_r(x))=0$.
By Vitali's covering theorem we can find countably many such balls, denoted by $(B_j)_{j\in\N}$, such that they are pairwise disjoint, have centers in $\Delta\cap(\Omega\times\Omega)$, and
\begin{equation*}
 (\Omega\times\Omega)\cap \Delta\subseteq N\cup\bigcup_j(B_j\times B_j)  \text{ with } \mu(N)=\lambda(PN)=0
\end{equation*}
(see Figure \ref{figureDeltaBj}).
By Proposition \ref{proplowerboundlocal} below applied with $\omega=B_j$ and using that 
$\mu(\partial(B_j\times B_j))=0$ we have
\begin{equation}\label{eqlambdamu}
 \lambda(B_j)\le \mu(B_j\times B_j).
\end{equation}
Then, denoting $\Delta_\eta:=\{(x,y)\in\R^2\times\R^2: |x-y|<2\eta\}$,
\begin{equation*}
 \lambda(\Omega)=\sum_{j\in\N}\lambda(B_j)\le \sum_{j\in\N} \mu(B_j\times B_j)\le \mu((\Omega\times\Omega)\cap\Delta_\eta).
\end{equation*}
Since this holds for any $\eta>0$, and $\mu(\Omega\times\Omega)<\infty$, we conclude
\begin{equation*}
 \lambda(\Omega)\le \mu((\Omega\times\Omega)\cap\Delta).
\end{equation*}
Recalling (\ref{eqLRMmu}) and that $\mu^{LR}\ll \calL^4$, we obtain $\mu_{LR}((\Omega\times\Omega)\cap\Delta)=0$ and
\begin{align*}
F_0[u,\Omega]=&
 \lambda(\Omega)+\mu_{LR}(\Omega\times\Omega) =
 \lambda(\Omega)+\mu_{LR}((\Omega\times\Omega)\setminus\Delta) \\
 \le& 
 \mu((\Omega\times\Omega)\cap\Delta)+\mu((\Omega\times\Omega)\setminus\Delta) = \mu(\Omega\times\Omega).
\end{align*}
This concludes the proof of (\ref{eqf0muomegaom}) and therefore of the proposition.
\end{proof}

It remains to prove the local lower bound stated in Proposition \ref{proplowerboundlocal}.
The proof uses a result from \cite{ContiGarroniMueller2011}  in which it is shown that  the nonlocal energy of almost-one-dimensional $BV$ phase fields controls the  line-tension energy of a similar field, that we recall in Proposition \ref{propconstruction}.
We start by fixing a mollifier $\varphi_0\in C^\infty_c(B_1;[0,\infty))$ with $\int_{B_1}\varphi_0dx=1$ and $\varphi\ge 1$ on $B_{1/2}$, and scaling it to
$\varphi_h(x):=2^{2h}\varphi_0(2^hx)$. We remark that the index $h$ in $\varphi_h$ denotes the exponent, at variance with the usage in the previous part of {this} paper, and recall the definition of the truncated energy in (\ref{eqdefestk}).

\begin{proposition}[\protect{\cite[Prop. 7.1]{ContiGarroniMueller2011}}]
\label{propconstruction}
Let $\omega\subset\subset\Omega$ be two bounded open sets, 
$u\in W^{1,1}(\Omega;\R^N)$,
$M>1$, $h,t\in\N$ with $t\ge 3$, $\zeta\in(0,1)$.
Assume $\dist(\omega,\partial\Omega)\ge 2^{-h+1}$.

Then there is $w=w_{h,\zeta,t,M}\in BV(\omega;\Z^N)$ such that
\begin{alignat}{1}
(\ln 2) \int_{J_w\cap \omega} \psi_\rel([w],\nu) \, d\calH^1 \le& 
(1+\zeta+c 2^{-t}) E^*_{h+t}[u, \Omega] 
+\frac{C_M}\zeta 2^{h+t} \|\dist(u,\Z^N)\|_{L^1(\Omega)}
\nonumber
\\
&+\frac{C_M}\zeta 2^{t} A^{5/6} 
\left(|Du|(\Omega) - |D(u \ast\varphi_{h})|(\omega)\right)^{1/6}
\nonumber
\\ 
&+ \frac{c}{M^{1/2}} 2^{t/2} A
\label{eqpropconstrlongestimateenergy}
\end{alignat}
and
\begin{alignat}{1}
\|u-w\|_{L^1(\omega)}\le& \frac{c}{M^{1/2}} 2^{-h+t/2} A
+ C_M \|\dist(u,\Z^N)\|_{L^1(\Omega)}
\nonumber
\\
&
+ C_M 2^{-h} A^{2/3} 
\left(|Du|(\Omega) - |D(u \ast\varphi_{h})|(\omega)\right)^{1/3}\,.
\label{eqpropconstrlongestimateLone}
\end{alignat}
Here $A:=\max\{|Du|(\Omega), E^*_{h+t}[u,\Omega]\}$. The constant $c$ may depend on $N$ and $\Gamma$, the constant $C_M$ also on $M$. 
\end{proposition}
We remark that the statement of Proposition 7.1 in \cite{ContiGarroniMueller2011} contains the unnecessary assumption that both sets are Lipschitz. The proof is based on covering $\omega$ with squares contained in $\Omega$ and performing a separate estimate on each square, in particular it never uses this assuption.

We finally give a proof of the lower bound in  Proposition~\ref{proplowerboundlocal}. 
The following argument is a modification of \cite[Prop. 8.1]{ContiGarroniMueller2011}.  It is here used only in the case that $\omega$ is a ball.

\begin{proof}[Proof of Proposition~\ref{proplowerboundlocal}]
It suffices to prove the estimate in the case $\omega=\Omega$ (otherwise we restrict all functions to $\omega$, and then  relabel $\omega$ as $\Omega$). We can also assume that the right-hand side is finite, and extract a subsequence such that the $\liminf$ is a limit.
We fix $\hat\omega\subset\subset\Omega$ and prove that
 \begin{equation}\label{eqeinftyeepsi}
  F_\self[u,\hat\omega] \le \liminf_{i\to\infty} \frac{E_{\eps_i}[u_i,\Omega]}{(\ln(1/\eps_i))^2}\,.
 \end{equation}
Taking the supremum over all such sets $\hat\omega$ will conclude the proof.

It remains to prove (\ref{eqeinftyeepsi}). To do this we choose a Lipschitz set $\Omega'$ such that
$\hat\omega\subset\subset\Omega'\subset\subset\Omega$ and 
fix $\delta>0$. By
Proposition \ref{propmakeubv} for  $i$  sufficiently large there are $k_i\in\N$ with
\begin{equation}\label{eqepsiki}
\e_i^{1-\delta/2}\le 2^{-k_i} \le \e_i^{1-\delta},
\end{equation}
which implies $(1-\delta)\ln\frac1{\eps_i}\le k_i\ln 2$, 
and
 a function
$v_{k_i,\delta}\in BV(\Omega';\Z^N)$ such that
\begin{equation}\label{eqvkideltauil1}
\lim_{i\to\infty} \|v_{k_i,\delta}-u_i\|_{L^1(\Omega')}
\le \lim_{i\to\infty}
C 2^{-k_i/2}\left( \ln \frac1{\eps_i}\right)^{1/2}=0
\end{equation}
and, with (\ref{eqsumehstar}),
\begin{equation}\label{eqliminfkliminfi}
\begin{split}
\liminf_{i\to\infty} \frac{1}{k_i^2} \sum_{h=0}^{k_i}
E^*_{h}[v_{k_i,\delta},\Omega'] 
\le& (\ln 2)^2\frac{1}{(1-\delta)^2} \liminf_{i\to\infty}
\frac{E_{\e_i}[u_i,\Omega]}{(\ln(1/\eps_i))^2} <\infty\,.
\end{split}
\end{equation}
With (\ref{eqpropmakeubvbvnorm}) we see that 
there is $A_\delta>0$ such that 
\begin{equation}\label{eqdefAdelta}
\frac{1}{k_i^2} \sum_{h=0}^{k_i}
E^*_h[v_{k_i,\delta},\Omega'] + \frac1{k_i} |Dv_{k_i,\delta}|(\Omega')\le A_\delta \hskip5mm\text{ for all $i\in\N$}\,.
\end{equation}
For simplicity of notation in the following we write $k$ for $k_i$, and correspondingly $\liminf_{k\to\infty}$ for $\liminf_{i\to\infty}$.

One important idea in the proof is to define an iterated mollification of the function $v_{k,\delta}$ using a family of length scales ranging from $1$ to $2^{-k}$. We use scales separated by a factor $2^m$,  in order to apply Proposition~\ref{propconstruction} between each pair of consecutive scales. The key idea is that each mollification step eliminates the structure present in the function on a certain length scale, as measured by the $BV$ norm. Since we have a $BV$ bound on the original function, and a large number of mollification steps, most of them will result in a very small reduction of the $BV$ norm, which means that on many scales the function will have an essentially one-dimensional structure. To make this precise, we fix
 $m\ge 3$ and define for $h\in\N$ the 
 sets
$\Omega_h:=\{x\in \R^2: {B_{2^{2-h}}}(x)\subseteq\Omega'\}$, so that
\begin{equation}\label{eqdistomegahhm}
\dist(\Omega_h, \partial\Omega_{h+m})\ge \dist(\Omega_h, \partial\Omega')-\dist(\Omega_{h+m},\partial\Omega')\ge {2^{1-h}} .
\end{equation}
 We then define for $h\in \N$ the function
 $z_{h}\in L^1(\Omega_h;\R^N)$ (implicitly depending also on $k$, $\delta$, and $m$)  by
\begin{equation*}
z_{h} := 
\begin{cases}
v_{k,\delta}, & \text { if } h\ge k\,, \\
z_{h+m}\ast \varphi_{h},& \text{ otherwise,}
\end{cases}
\end{equation*}
where $\varphi_h$ is the mollifier that enters  Proposition~\ref{propconstruction}.

One key estimate, which is obtained by summing the $m$ telescoping series and using (\ref{eqdefAdelta}), is 
\begin{equation}\label{eqdzsumadelta}
 \begin{split}
\sum_{h=0}^k \left[ |Dz_{h+m}|(\Omega_{h+m}) - |Dz_{h}|(\Omega_h)\right] 
&= \sum_{h=k+1}^{k+m}|Dz_{h}|(\Omega_h) - 
\sum_{h=0}^{m-1}|Dz_{h}|(\Omega_h)\\
& \le  m |Dv_{k,\delta}|(\Omega')\le   k  m A_\delta\,.  
 \end{split}
\end{equation}
By the properties of the mollification we also obtain, {for $h< k$,}
\begin{alignat*}{1}
\|z_h-z_{h+m}\|_{L^1(\Omega_h)}=\|z_{h+m}*\varphi_h-z_{h+m}\|_{L^1(\Omega_h)}&\le  2^{-h} |Dz_{h+m}|(\Omega_{h+m}) \\
&\le 
2^{-h} |Dv_{k,\delta}|(\Omega')\le  k 2^{-h} A_\delta \,, 
\end{alignat*}
and therefore
\begin{equation}\label{eql1zhvkd}
\|z_{h} -v_{k,\delta}\|_{L^1(\Omega_h)} \le 
\sum_{j=0}^\infty 
\|z_{h+jm} -z_{h+(j+1)m}\|_{L^1(\Omega_h)} \le 
2 k 2^{-h} A_\delta\,. 
\end{equation}
Since $v_{k,\delta}\in\Z^N$ a.e., this implies
\begin{equation}\label{eqdistuhZN}
\|\dist(z_{h}, \Z^N)\|_{L^1(\Omega_h)}\le 
2 k 2^{-h} A_\delta \,. 
\end{equation}
Convexity and translation invariance of
the nonlocal energy imply that mollification decreases the energy, and using (\ref{eqdistomegahhm}) we have
\begin{equation*}
E^*_{s}[u*\varphi_{h},\Omega_h]\leq E^*_{s}[u,\Omega_{h+m}] \text{ for any }s\in\N,\,\text{ and } u\in L^2(\Omega_{h+m};\R^N)
\end{equation*}
and therefore,  iterating this inequality,
\begin{equation*}
 E^*_{s}[z_{h},\Omega_h]\leq E^*_{s}[v_{k,\delta},\Omega']
\end{equation*}
for any $s$ and $h$. In particular,
\begin{equation}\label{eqehtehtvj}
E^*_{h+t}[z_{h+m},\Omega_{h+m}]\leq E^*_{h+t}[v_{k,\delta},\Omega']. 
\end{equation}

At this point we choose $\zeta\in(0,1/4)$ and $t\in\N$, with $m\ge t\ge 3$.
Since we shall take the limit $k\to\infty$ first, we can assume that $k\ge m/\zeta$. 
We now choose a good value for $h\in (\zeta
k,k-\zeta k)\cap \N$. Specifically, let
\begin{equation*}
 J:=\{h\in (\zeta
k,k-\zeta k)\cap \N: E^*_{h+t}[v_{k,\delta},\Omega']>(1+5\zeta) \frac{1}{k} \sum_{j=0}^k
E^*_j[v_{k,\delta},\Omega']\}
\end{equation*}
and
\begin{equation*}
 H:=\{h\in (\zeta
k,k-\zeta k)\cap \N: |Dz_{h+m}|(\Omega_{h+m})-
|Dz_{h}|(\Omega_h) 
> \frac{m}{\zeta} A_\delta\}.
\end{equation*}
One easily verifies that
\begin{equation*}
 \# J \frac{1+5\zeta}{k}  \le 1
\end{equation*}
and, recalling (\ref{eqdzsumadelta}),
\begin{equation*}
 \# H \frac{1}{\zeta} \le k.
\end{equation*}
We assume $\zeta\le 1/20$, which implies $1/(1+5\zeta)\le 1-4\zeta$, and obtain, since $\zeta k\ge m\ge 3$,
\begin{equation*}
 \#J+\#H \le \frac{k}{1+5\zeta} +k\zeta \le (1-2\zeta)k -3.
\end{equation*}
Since $\#( (\zeta
k,k-\zeta k)\cap \N)\ge (1-2\zeta)k-2$, 
this implies that we can choose $h\in  (\zeta
k,k-\zeta k)\cap \N\setminus (J\cup H)$. This value will be fixed for the rest of the argument (depending on the other parameters) and satisfies, 
recalling (\ref{eqehtehtvj}),
\begin{equation}\label{eqhgoodp}
E^*_{h+t}[z_{h+m},\Omega_{h+m}]\le E^*_{h+t}[v_{k,\delta},\Omega'] \le (1+5\zeta) \frac{1}{k} \sum_{j=0}^k
E^*_j[v_{k,\delta},\Omega']
\end{equation}
and
\begin{equation}\label{eqhgoodpdue}
|Dz_{h+m}|(\Omega_{h+m})-
|Dz_{h}|(\Omega_h) 
\le  \frac{m}{\zeta} A_\delta\,.
\end{equation}

We apply  Proposition~\ref{propconstruction} to
$z_{h+m}$, for some $M>1$ chosen below, on the sets $\Omega_h\subset\subset\Omega_{h+m}$  and denote the
result by $w=w_{k,\delta,m,t,M}$.  We obtain
\begin{alignat}{1}
(\ln 2) \int_{J_w\cap \Omega_h} \psi_\rel([w],\nu) \, d\calH^1 \le& 
(1+\zeta+c 2^{-t}) E^*_{h+t}[z_{h+m}, \Omega_{h+m}] \nonumber\\
&+\frac{C_M}{\zeta} 2^{h+t} \|\dist(z_{h+m},\Z^N)\|_{L^1(\Omega_{h+m})}
\nonumber
\\
&+\frac{C_M}{\zeta} 2^{t} A_*^{5/6} 
\left(|Dz_{h+m}|(\Omega_{h+m}) - |Dz_{h}|(\Omega_h)\right)^{1/6}
\nonumber
\\ 
&+ \frac{c}{M^{1/2}} 2^{t/2} A_*
\label{eqprop33hm}
\end{alignat}
and
\begin{alignat}{1}
\|z_{h+m}-w\|_{L^1(\Omega_h)}\le& \frac{c}{M^{1/2}} 2^{-h+t/2} A_*
+ C_M \|\dist(z_{h+m},\Z^N)\|_{L^1(\Omega_{h+m})}
\nonumber
\\
&
+ C_M 2^{-h} A_*^{2/3} 
\left(|Dz_{h+m}|(\Omega_{h+m}) - |Dz_h|(\Omega_h)\right)^{1/3}\,,
\label{eqprop33l1}
\end{alignat}
where $A_*:=\max\{|Dz_{h+m}|(\Omega_{h+m}) , E^*_{h+t}[z_{h+m},\Omega_{h+m}]\}$.
Recalling (\ref{eqhgoodp}),
$|Dz_{h+m}|(\Omega_{h+m})\le |Dv_{k,\delta}|(\Omega')$ and then 
(\ref{eqdefAdelta}), we obtain
\begin{equation}\label{eqAstAdelta}
 A_*\le  |Dv_{k,\delta}|(\Omega') + \frac2k\sum_{j=0}^k
E^*_j[v_{k,\delta},\Omega']\le 2kA_\delta.
\end{equation}
Then (\ref{eqprop33l1}) becomes, using (\ref{eqdistuhZN}) and (\ref{eqhgoodpdue}),
\begin{alignat}{1}
\|z_{h+m}-w\|_{L^1(\Omega_{h})}\le& \frac{c}{M^{1/2}} 2^{-h+t/2} k A_\delta
+ C_M  k 2^{-h-m} A_\delta 
\nonumber
\\ &
+ C_M 2^{-h} k^{2/3}A_\delta m^{1/3} \zeta^{-1/3}\,.
\label{eqprop33l1due}
\end{alignat}
We recall that $\hat\omega\subset\Omega_h$, for sufficiently large $k$, since we chose $h\ge k\zeta$.
From
(\ref{eqprop33hm}), (\ref{eqhgoodp}), (\ref{eqdistuhZN}), and (\ref{eqhgoodpdue}),
\begin{equation}\label{eqprop33hm2}
 \begin{split}
\frac{\ln 2}{k} \int_{J_w\cap \hat\omega} \psi_\rel([w],\nu) \, d\calH^1 \le& 
(1+\zeta+c 2^{-t}) 
 (1+5\zeta) \frac{1}{k^2} \sum_{j=0}^k
E^*_j[v_{k,\delta},\Omega']
\\
&+\frac{C_M}{\zeta} 2^{h+t} 2^{-h-m} A_\delta
\\
&+\frac{C_M}{k\zeta} 2^{t} A_*^{5/6} 
\left(m\zeta^{-1} A_\delta\right)^{1/6}
\\ 
&+ \frac{c}{kM^{1/2}} 2^{t/2} A_*
   \end{split}
\end{equation}
We notice that this expression does not depend any more explicitly on the choice of $h$, since $ 2^{h+t} 2^{-h-m}= 2^{t-m}$.

We set $u^k:=\frac1{k\ln 2}w$, where $w=w_{k,\delta,m,t,M}$
is the function constructed in Proposition~\ref{propconstruction}, 
so that the relaxed line-tension functional $\ELTrel_\sigma$ defined in (\ref{eqdefeltrel}) reads as
\begin{equation*}
 \ELTrel_{1/(k\ln 2)}[u^k,\hat\omega] = \frac1{k\ln 2}\int_{J_{w}\cap \hat\omega} \psi_\rel([w],\nu) \, d\calH^1 .
\end{equation*}
Equation (\ref{eqprop33hm2}), together with (\ref{eqdefAdelta}) 
and (\ref{eqAstAdelta}) 
then yields, for sufficiently large $k$,  
 \begin{alignat}{1}
\ELTrel_{1/(k\ln 2)}[u^k,\hat\omega]
 &  \le \frac{1}{k^2 (\ln 2)^2} \sum_{j=0}^k
E^*_{j}[v_{k,\delta},\Omega'] + 
(c\zeta+c 2^{-t}) A_\delta
+\frac{C_M}\zeta 2^{t-m} A_\delta
\nonumber
\\
&+\frac{C_M}\zeta 2^{t} A_\delta^{5/6} 
\left(\frac{m}{k \zeta} A_\delta\right)^{1/6}+ \frac{c}{M^{1/2}} 2^{t/2} A_\delta.
\label{ELTrelln2k}
\end{alignat}
Correspondingly, from (\ref{eqprop33l1due}) a similar procedure leads to
\begin{alignat}{1}
\|\frac1{k\ln 2} z_{h+m}-u^k\|_{L^1(\hat\omega)}\le& \frac{c}{M^{1/2}} 2^{-h+t/2} A_\delta
+ C_M   2^{-h-m} A_\delta 
\nonumber
\\ &
+ C_M 2^{-h} k^{-1/3}A_\delta m^{1/3} \zeta^{-1/3}\,.
\label{eql1ln2lk}
\end{alignat}
By (\ref{ELTrelln2k}) and (\ref{eqdefAdelta}) we obtain that $\limsup_{k\to\infty} \ELTrel_{1/(k\ln 2)}[u^k,\hat\omega]
<\infty$. 
Recalling  {$\ELTrel_\sigma\le \ELT_\sigma$ and} the compactness statement in Theorem \ref{theo:homog}, there are $d_k\in\R^N$ such that, after extracting a subsequence, $u^k-d_k$ 
converges as $k\to\infty$ to some
 $u^{\delta,m,t,M}$ in $L^1(\hat\omega;\R^N)$. 
Taking the limit $k\to\infty$, and recalling Theorem 
\ref{theo:homog} and (\ref{eqliminfkliminfi}) we obtain
 \begin{alignat}{1}
\ELT_0[u^{\delta,m,t,M},\hat\omega]
   \le& \frac{1}{(1-\delta)^2} \liminf_{i\to\infty}
\frac{E_{\e_i}[u_i,\Omega]}{(\ln(1/\eps_i))^2}  + 
(c\zeta+c 2^{-t}) A_\delta
\nonumber\\
&+\frac{C_M}\zeta 2^{t-m} A_\delta+ \frac{c}{M^{1/2}} 2^{t/2} A_\delta.
\label{eqfselfudmtM}
\end{alignat}
At the same time by (\ref{eql1ln2lk}) we have
\begin{alignat*}{1}
\limsup_{k\to\infty} \|\frac1{k\ln 2} z_{h+m}-u^k\|_{L^1(\hat\omega)}\le& \frac{c}{M^{1/2}} 2^{-h+t/2} A_\delta
+ C_M   2^{-h-m} A_\delta \,.
\end{alignat*}
By (\ref{eql1zhvkd}) we have
\begin{equation*}
\|\frac1{k\ln 2}z_{h+m} -\frac1{k\ln 2} v_{k,\delta}\|_{L^1(\hat\omega)} \le 
C  2^{-h-m} A_\delta\,,
\end{equation*}
and therefore
\begin{equation*}
\limsup_{k\to\infty}\|u^k -\frac1{k\ln 2} v_{k,\delta}\|_{L^1(\hat\omega)} \le \frac{c}{M^{1/2}} 2^{-h+t/2} A_\delta
+ C_M   2^{-h-m} A_\delta \,.
\end{equation*}
With (\ref{eqvkideltauil1}), and going back to the notation where the index $i$ is explicit, we obtain
\begin{equation*}
\limsup_{i\to\infty}\|u^{k_i} -\frac1{k_i\ln 2} u_i\|_{L^1(\hat\omega)} \le \frac{c}{M^{1/2}} 2^{-h+t/2} A_\delta
+ C_M   2^{-h-m} A_\delta \,,
\end{equation*}
so that (\ref{eqepsiki}) gives 
\begin{equation*}
\begin{split}
 \limsup_{i\to\infty}\|u^{k_i} -\frac{1}{\ln (1/\eps_i)} 
u_i\|_{L^1(\hat\omega)} 
\le& \frac{c}{M^{1/2}} 2^{-h+t/2} A_\delta
\\
&+ C_M   2^{-h-m} A_\delta +c\delta \limsup_{i\to\infty}\|\frac{1}{\ln 
(1/\eps_i)} u_i\|_{L^1(\hat\omega)}\,,
\end{split}
\end{equation*}
and since $u_i/\ln(1/\eps_i)\to u$ and $u^{k_i}\to u^{\delta,m,t,M} $ in $L^2(\hat\omega)$,
\begin{equation*}
\|u^{\delta,m,t,M} -u\|_{L^1(\hat\omega)} \le \frac{c}{M^{1/2}} 2^{-h+t/2} A_\delta
+ C_M   2^{-h-m} A_\delta
+c\delta\|u\|_{L^1(\hat\omega)}\,.
\end{equation*}
The argument is then concluded recalling (\ref{eqfselfudmtM}) and taking a suitable diagonal subsequence. Indeed, 
as $\delta$, $m$, $\zeta$, $M$ and $t$ were arbitrary, and
since 
$u^{\delta,m,t,M}\to u$, by lower semicontinuity
of $F_\self$ taking first $m\to\infty$, then $\zeta\to0$, then $M\to\infty$, then $t\to\infty$, and finally $\delta\to0$, we conclude
 \begin{alignat*}{1}
F_\self[u_0,\hat\omega]=\ELT_0[u,\hat\omega]
   \le&  \liminf_{i\to\infty}
\frac{E_{\e_i}[u_i,\Omega]}{(\ln(1/\eps_i))^2}  .
\end{alignat*}
This concludes the proof of (\ref{eqeinftyeepsi}) and therefore of the Proposition.
\end{proof}


\section*{Acknowledgements}
This work was partially funded by the Deutsche Forschungsgemeinschaft (DFG, German Research Foundation) through project 211504053 -- SFB 1060
and project 390685813 -- GZ 2047/1.

\bibliographystyle{alpha-noname}
\bibliography{Co-Gar-Mu-20}

\end{document}